\pdfoutput=1
\documentclass[microtype]{gtpart}
\usepackage{graphicx}
\usepackage[mathscr]{eucal}
\usepackage{amssymb}
\usepackage[dvipsnames]{xcolor}
\usepackage{enumerate}
\usepackage[noadjust]{cite}
\usepackage{overpic}
\usepackage{pinlabel}
\usepackage[margin=1.4in]{geometry}
\usepackage{mathabx} 

\newcommand{\br}{\mathbb R}

\newcommand{\bz}{\mathbb Z}
\newcommand{\bn}{\mathbb N}
\newcommand{\bq}{\mathbb Q}

\newcommand{\cE}{\mathcal E}
\newcommand{\cc}{\mathcal C}

\newcommand{\sU}{\mathscr U}
\newcommand{\sV}{\mathscr V}

\newcommand{\vp}{\varphi}

\newcommand{\Ht}{\Gamma}

\newcommand{\ssm}{\smallsetminus}

\DeclareMathOperator{\mcg}{MCG}

\DeclareMathOperator{\Ends}{\cE}

\DeclareMathOperator{\Homeo}{Homeo}

\DeclareMathOperator{\supp}{supp}

\renewcommand{\co}{\colon\thinspace}

\newtheorem{Thm}{Theorem}[section]
\newtheorem{Thm*}{Theorem}
\newtheorem{Prop}[Thm]{Proposition}
\newtheorem{Lem}[Thm]{Lemma}
\newtheorem{Cor}[Thm]{Corollary}
\newtheorem{Cor*}[Thm*]{Corollary}

\newtheorem*{MainThm2}{Theorem~\ref{thm:characterize}}

\theoremstyle{definition}
\newtheorem{Def}[Thm]{Definition}

\numberwithin{equation}{section}

\title{Homeomorphism groups of 2-manifolds \\with the virtual Rokhlin property}

\author{Justin Lanier}
\address{School of Mathematics and Statistics F07 \\ The University of Sydney \\ NSW 2006 Australia} 
\email{justin.lanier@sydney.edu.au}

\author{Nicholas G.~Vlamis}
\address{Department of Mathematics \\ CUNY Graduate Center \\ New York, NY 10016, and \newline Department of Mathematics \\ CUNY Queens College \\ Flushing, NY 11367}
\email{nicholas.vlamis@qc.cuny.edu}

\begin{document}  

\begin{abstract}
We introduce and motivate the definition of the virtual Rokhlin property for topological groups.
We then classify the 2-manifolds whose homeomorphism groups have the virtual Rokhlin property.
We also establish the analogous result for mapping class groups of 2-manifolds. 
\end{abstract}

\maketitle

\vspace{-0.5in}


\section{Introduction}

A topological group has the \emph{Rokhlin property} if there exists an element of the group whose conjugacy class is dense.
In \cite{LanierMapping}, we gave a complete characterization of the orientable 2-manifolds whose mapping class groups have the Rokhlin property.
This result was independently and simultaneously obtained in \cite{HHMRSV}. 
In our previous work, the mapping class group of a subclass of 2-manifolds failed to have the Rokhlin property due to the existence of an open index two subgroup---a more trivial reason than most other cases.
The subclass consists of 2-manifolds with exactly two maximal ends such that these ends can be permuted; examples include the open annulus and the ladder surface.
The origin of this article is our interest in characterizing when these subgroups have the Rokhlin property themselves. 
One natural reason to address this problem is the fact that the algebraic consequence of the Rokhlin property observed in  our first article continues to hold when a finite-index open normal subgroup has the Rokhlin property (this is discussed further below). 
With this in mind, we introduce the following definition:

\begin{Def}[Virtual Rokhlin property]
A topological group has the \emph{virtual Rokhlin property} if it contains an open finite-index subgroup with the Rokhlin property (where the subgroup is considered with the subspace topology). 
\end{Def}

Our main theorem is a characterization of the 2-manifolds whose homeomorphism groups have the viritual Rokhlin property in terms of properties of the topology of the 2-manifold. 
Here, the homeomorphism group of a 2-manifold \( M \), denoted \( \Homeo(M) \), is the group consisting of homeomorphisms \( M \to M \), equipped with the compact-open topology and viewed as a topological group.
For the characterization, we introduce the notion of a \emph{spacious} 2-manifold, which---for the reader familiar with the work of Mann--Rafi \cite{MannRafi}---is meant to capture some of the structure of telescoping 2-manifolds with exactly two maximal ends.

The remainder of the introduction is divided into several subsections.
We begin by introducing the requisite definitions and stating the main theorem and then proceed to several subsections dedicated to corollaries and results of independent interest.
In particular, we will classify the finite-type 2-manifolds \( M \) for which \( \Homeo_0(M) \) has the Rokhlin property, describe how the virtual Rokhlin property obstructs the existence of certain countable quotients, and explain how the virtual Rokhlin property limits the possibilities for ``interesting'' group actions.

\subsection*{Definitions and main result}

A subset \( K \) of a topological space \( X \) is \emph{displaceable} if there exists a homeomorphism \( f \co X \to X \) such that \( f(K) \cap K = \varnothing \). 
A 2-manifold \( M \) is \emph{weakly self-similar} if every proper compact subset of \( M \) is displaceable; \( M \) is \emph{self-similar} if it is weakly self-similar and every separating simple closed curve in \( M \) has a complementary component that is homeomorphic to \( M \) with a point removed.

By the Schoenflies theorem, the 2-sphere and the plane are self-similar; in fact, they are the only finite-type self-similar 2-manifolds. 
The definition of self-similar given above is introduced in \cite{VlamisHomeomorphism} as an end-free alternative to the original definition given in \cite{MalesteinSelf}.  
We also note that the open annulus \( \mathbb S^1 \times \br \) is the only finite-type 2-manifold that is weakly self-similar but not self-similar. 
Therefore, given a 2-manifold \( M \) that is weakly self-similar but not self-similar, \( M \) is non-compact, which allows us to apply  \cite[Theorem~2.4~and~Lemma~5.4]{LanierMapping} to see that \( M \) is homeomorphic to \( N \#N \) for some  self-similar 2-manifold \( N \).  

Let \( \Sigma \) be a subsurface of a 2-manifold \( M \) such that \( \Sigma \) is closed as a subset of \( M \) and has compact boundary.
Then, the inclusion of \( \Sigma \) into \( M \) induces an open embedding of the end space of \( \Sigma \) into the end space of \( M \). 
Given an end \( e \) in \( M \), we say \( \Sigma \) is a \emph{neighborhood of \( e \) in \( M \)} if \( e \) is in the end space of \( \Sigma \), viewed as a subset of the end space of \( M \).

\begin{Def}[Maximally stable end]
Let \( M \) be a weakly self-similar 2-manifold, and let \( N \) be a self-similar 2-manifold such that \( M \) is homeomorphic to \( N \) or \( N \#N \).
An end \( \mu \) of \( M \) is \emph{maximally stable} if there exists a nested sequence of subsurfaces \( \{\Sigma_n\}_{n\in\bn} \) such that \( \Sigma_n \) is a neighborhood of \( \mu \), \( \Sigma_n \) is homeomorphic to \( N \) with an open disk removed, and \( \bigcap_{n\in\bn} \Sigma_n  = \varnothing \). 
\end{Def}

Note that every maximally stable end is both \emph{maximal} and \emph{stable} in the sense of Mann--Rafi \cite{MannRafi}, and in fact, every maximal end of a weakly self-similar 2-manifold is maximally stable.
The Mann--Rafi notion of maximal end is defined in terms of a preorder on the end space of the manifold, which we do not require for our main theorem, but we will use---and introduce---the preorder in Section~\ref{sec:detection}.

The set of maximally stable ends of a weakly self-similar 2-manifold is either a singleton, doubleton, or is perfect (see Proposition~\ref{prop:ss-partition} below). 
Moreover, a self-similar 2-manifold either has a unique maximally stable end or a perfect set of maximally stable ends; in the former case, the manifold is said to be \emph{uniquely self-similar}, and in the latter, is said to be \emph{perfectly self-similar}.

The \emph{support} of a homeomorphism \( f \co M \to M \), denoted \( \supp(f) \), is the closure of the set \( \{x\in M: f(x) \neq x\} \). 
Given a weakly self-similar 2-manifold \( M \), let \( \Ht(M) \) denote the subgroup of \( \Homeo(M) \) consisting of homeomorphisms supported outside a neighborhood of the set of maximally stable ends of \( M \). 
Observe that, as \( \mathbb S^2 \) has no ends, \( \Ht(\mathbb S^2) = \Homeo(\mathbb S^2) \).

\begin{Def}[Spacious 2-manifold]
A 2-manifold \( M \) is \emph{spacious} if it is weakly self-similar and, given any simple closed curve \( c \) in \( M \), the action of \( \Ht(M) \) on  \( \Homeo(M) \cdot c \) has finitely many orbit classes.
\end{Def}

Throughout the article, we denote the orbit classes of the action of \( \Gamma(M) \) on \( \Homeo(M) \cdot c \) by \( \Gamma(M) \backslash \Homeo(M) \cdot c \). 

Let us consider some examples.
The 2-sphere is spacious and is the only compact spacious manifold. 
As the 2-sphere is a rather special case, the two key motivating examples of spacious 2-manifolds are the plane \( \br^2 \) and the open annulus \( \mathbb S^1 \times \br \). 
An uncountable family of examples is obtained by deleting a set from the 2-sphere homeomorphic to the ordinal space \( \omega^\alpha\cdot n + 1 \), where \( \omega \) is the first countable ordinal, \( n \in \{1,2\} \), and \( \alpha \) is any countable ordinal if \( n =1 \) and any countable limit ordinal if \( n = 2 \). 
Note that in the above example, if \( n =2 \) and \( \alpha \) were a successor ordinal, then the resulting manifold would be weakly self-similar, but not spacious. 
Similarly, the the ladder surface---the orientable two-ended infinite-genus 2-manifold with no planar ends---fails to be spacious: any two non-homotopic disjoint simple closed curves separating the two ends are in different \( \Ht \)-orbits (this was a key observation in \cite{PatelVlamis}, which consequently led to the results in \cite{APV}). 

We can now state the main theorem:
\vspace{-6pt}
\begin{Thm}
\label{thm:main1}
The homeomorphism group of a 2-manifold \( M \) has the virtual Rokhlin property if and only if \( M \) is spacious.
Moreover, if \( M \) is spacious, then the closure of \( \Ht(M) \) in \( \Homeo(M) \) is an open finite-index subgroup with the Rokhlin property.
\end{Thm}
\vspace{-6pt}
The discussion thus far has been about homeomorphism groups, so let us state a version of Theorem~\ref{thm:main1} that pertains to mapping class groups.
To simplify statements, we introduce the following notation:


Let \( \Homeo_0(M) \) denote the connected component of the identity in \( \Homeo(M) \) with respect to the compact-open topology on \( \Homeo(M) \). 
If \( M \) is orientable, then the \emph{mapping class group} of \( M \), denoted \( \mcg(M) \), is defined to be \( \Homeo^+(M)/\Homeo_0(M) \); if \( M \) is not orientable, \( \mcg(M) \) is defined to be \( \Homeo(M)/ \Homeo_0(M) \).

\begin{Thm}
\label{thm:main2}
Let \( M \) be a 2-manifold with infinite mapping class group.
Then, the following are equivalent:
\begin{enumerate}
\item \( \Homeo(M) \) has the virtual Rokhlin property.
\item \( \mcg(M) \) has the virtual Rokhlin property.
\item \( M \) is spacious. 
\end{enumerate}
\end{Thm}

Though topological, the definition of spaciousness is a bit opaque, and we would prefer to have a characterization that relies on the data provided by the classification of 2-manifolds, namely the genus and the topology of the end space.  
As we will see in Theorem~\ref{thm:import} (as a consequence of our previous work), we can characterize the self-similar 2-manifolds that are spacious, namely they are the uniquely self-similar 2-manifolds.  
As every weakly self-similar 2-manifold that fails to be self-similar has two maximally stable ends (see Proposition~\ref{prop:ss-partition}), it is left to characterize the weakly self-similar 2-manifolds with exactly two maximally stable ends that are spacious.
This  is accomplished in the next theorem under a tameness condition on the end space, which we call \emph{upward stability} (see Definition~\ref{def:upward_stable} below).
We will see in Proposition~\ref{prop:monster} that the characterization, as stated, does in fact require some type of tameness condition. 

In the theorem statement, immediate predecessor is meant in the sense of the preorder of Mann--Rafi; see Section~\ref{sec:detection} for details.

\begin{Thm}
\label{thm:characterize}
Let \( M \) be an upwardly stable weakly self-similar 2-manifold with exactly two maximally stable ends.
Then, \( M \) is spacious if and only if none of the following conditions are met:
\begin{enumerate} [(i)]
\item a maximally stable end of \( M \) has an immediate predecessor with countable \( \Homeo(M) \)-orbit.
\item \( M \) is orientable, infinite genus, and only the maximal ends are non-planar. 
\item \( M \) is not orientable, infinite genus, and only the maximal ends are not orientable. 
\end{enumerate}
\end{Thm}

In each of the cases above in which \( M \) fails to be spacious, there is something to count, yielding the following corollary.

\begin{Cor}
\label{cor:cohomology}
Let \( M \) be a upwardly stable weakly self-similar 2-manifold with exactly two maximally stable ends, and let \( G \) denote the stabilizer of the maximal ends in \( \Homeo(M) \).
Then, \( M \) is spacious if and only if every homomorphism from \( G \) to \( \bz \) is trivial, or equivalently \( H^1(G,\bz) \) is trivial.
\end{Cor}

\begin{proof}
Suppose \( M \) fails to be spacious. 
Let \( c \) be an oriented simple closed curve separating the maximally stable ends of \( M \). 
By Theorem~\ref{thm:characterize}, given \( f \in \Homeo(M) \) that stabilizes the maximal ends, we can count the net difference of (i) immediate predecessors, (ii) handles, or (iii) crosscaps \( f \) moves from the left of \( c \) to the right. 
This type of counting has appeared several times in the literature (e.g., \cite{APV, APV2, MannRafi, HernandezFirst}) and is known to yield a homomorphism from the stabilizer of the maximally stable ends in \( \Homeo(M) \) to \( \bz \).

For the converse, suppose that \( M \) is spacious.
By Corollary~\ref{thm:main1}, \( \Homeo(M) \) has the virtual Rokhlin property, and by Propositon~\ref{prop:vr-normal} below, \( G \) has the virtual Rokhlin property.  
Therefore, by Proposition~\ref{prop:trivial} below, every homomorphism \( G \to \bz \) is trivial. 
\end{proof}

We note that if \( M \) is as in Corolllary~\ref{cor:cohomology}, it is possible for \( H_1(G, \bz) \) (i.e, the abelianization of \( G \)) to be large even though \( H^1(G, \bz) \) is trivial. For example, Domat--Dickmann \cite{DomatBig}  showed if \( M \) is the Loch Ness monster surface, then \( H_1(\Homeo(M), \bz) \) contains uncountably many copies of \( \bq \).

\subsection*{A result about finite-type 2-manifolds}

The condition on the cardinality of the mapping class group in Theorem~\ref{thm:main2} is necessary as there exist 2-manifolds with finite mapping class groups and that fail to be spacious
 (namely, the thrice-punctured sphere, the projective plane, the Klein bottle, the open M\"obius band, and the twice-punctured projective plane), and every finite group has the virtual Rokhlin property.
In fact,  the homeomorphism groups of non-spacious 2-manifolds with finite mapping class groups fail to have the virtual Rokhlin property. 
This is a consequence of Theorem~\ref{thm:main3}, which is of independent interest and is established in the process of proving Theorem~\ref{thm:main1}.

\begin{Thm}
\label{thm:main3}
Let \( M \) be a finite-type 2-manifold.
Then \( \Homeo_0(M) \) has the virtual Rokhlin property if and only if \( M \) is homeomorphic to the 2-sphere, plane, or open annulus.
Moreover, if \( M \) is homeomorphic to the 2-sphere, plane, or open annulus, then \( \Homeo_0(M) \) has the Rokhlin property; in particular, \( \Homeo_0(M) \) has the Rokhlin property if and only if it has the virtual Rokhlin property.
\end{Thm}

The fact that \( \Homeo_0(\mathbb S^2) \) has the Rokhlin property was proved by Glasner--Weiss \cite{GW}, a proof of the fact that \( \Homeo_0(\mathbb R^2) \) has the Rokhlin property can be found in \cite{VlamisHomeomorphism}, and the fact that \( \Homeo_0(\mathbb S^1 \times \br) \) has the Rokhlin property is a corollary of Theorem~\ref{thm:main1}. 
We give a direct proof that \( \Homeo_0(M) \) cannot have the Rokhlin property for any other finite-type 2-manifold \( M \), see Lemma~\ref{lem:homeo0_no} and Lemma~\ref{lem:homeo0_o}.
For a compact 2-manifold \( M \) of genus at least one, one can deduce that \( \Homeo_0(M) \) fails to have the Rokhlin property using the work of Bowden--Hensel--Mann--Militon--Webb \cite[Theorem~1.2]{BowdenRotation}. They show that the asymptotic translation length of the action of \( \Homeo_0(M) \) on the fine curve graph of \( M \) yields a continuous (conjugation-invariant) surjection onto the non-negative reals, and as explained in the next paragraph, such a function could not exist if \( \Homeo_0(M) \) had the Rokhlin property. 

\subsection*{Geometric consequence}

Suppose \( G \) is a topological group with the Rokhlin property and that \( f \co G \to \br \) is a continuous conjugation-invariant function. 
If the conjugacy class of \( g \in G \) is dense, then \( f(h) = f(g) \) for each \( h \in G \); in particular, \( f \) is constant.
We can use this observation to greatly restrict the geometry of topological groups with the virtual Rokhlin property, which we record in the following corollary for homeomorphism groups of spacious 2-manifolds.

\begin{Cor}
\label{cor:geometric}
Let \( M \) be a spacious 2-manifold.
Then: 
\begin{enumerate}
\item Every continuous bi-invariant metric on \( \Homeo(M) \) is bounded.
\item Every continuous  quasimorphism of \( \Homeo(M) \) is bounded. 
\item If \( \Homeo(M) \) acts continuously on a Gromov hyperbolic space by isometries, then the asymptotic translation length of every element is 0; in particular, no element of \( \Homeo(M) \) acts as a hyperbolic isometry on a Gromov hyperbolic metric space under a continuous action. \qed
\end{enumerate}
\end{Cor}

Corollary~\ref{cor:geometric}(1) is a direct consequence of the preceding discussion by considering the function \( f\co \Homeo(M) \to \br \) given by \( f(g) = d( id, g) \), where \( d \) is the metric. 
Corollary~\ref{cor:geometric}(2) is deduced from standard facts about quasimorphisms (see \cite{scl}).
Corollary~\ref{cor:geometric}(3) uses the limit definition of asymptotic translation length (see \cite[II.6.6]{Bridson}) and the fact that if a group acts continuously on a Gromov hyperbolic metric space, then the associated asymptotic translation length function on the group is continuous (see the sequence of lemmas establishing \cite[Theorem~3.4]{BowdenRotation}\footnote{The sequence of lemmas is stated in the context of a particular action, but they hold in the general setting from Corollary~\ref{cor:geometric}, and in fact the continuity hypothesis can be weakened (c.f. \cite[Lemma~3.5]{BowdenRotation}).}). 

Corollary~\ref{cor:geometric} is meant to highlight the fact that the virtual Rokhlin property is a boundedness property and that it limits the potential interesting geometric actions a group can admit. 
We encourage the reader to compare Corollary~\ref{cor:geometric}(1) to the bounded groups of Burago--Ivanov--Polterovich \cite{BuragoConjugation}, i.e., groups in which every bi-invariant metric is bounded.

Here is one instance of Corollary~\ref{cor:geometric}(3): if \( M \) is spacious and \( \Homeo(M) \) acts on a tree continuously, then every element of \( \Homeo(M) \) acts elliptically (as trees do not have parabolic isometries); this is weaker than, but has a similar flavor to Serre's property (FA).
Recall, a group has \emph{property (FA)} if every action of the group on a tree has a global fixed point (see \cite{SerreTrees}). 
It is readily seen that a group with property (FA)  cannot surject onto \( \bz \); this is an example of the geometry of a group informing its algebraic structure. 
Below, we will also see a similar link for the virtual Rokhlin property; in particular, we see in Corollary~\ref{cor:homomorphism}, that if \( M \) is spacious, then \( \Homeo(M) \) cannot surject onto \( \bz \), an instance of the topology of the group informing its algebraic structure.

To finish this subsection, we note that boundedness properties of homeomorphism groups and mapping class groups of 2-manifolds have been previously studied.
For instance, one of Mann--Rafi's results in \cite{MannRafi} have shown \( \mcg(M) \) is coarsely bounded if \( M \) is self-similar or teslescoping, and the second author has shown that \( \Homeo(M) \) is strongly bounded if \( M \) is telescoping \cite{VlamisHomeomorphism, VlamisTelescoping}.

\subsection*{Algebraic consequence}

When combined with automatic continuity properties, the Rokhlin property can have algebraic consequences, for instance obstructing the existence of abstract homomorphisms to a large class of countable groups.
A group \( H \) is \emph{cm-slender} if the kernel of any abstract group homomorphism \( G \to H \) is open whenever \( G \) is completely metrizable.  
The class of such groups is quite large and includes free abelian groups, torsion-free word-hyperbolic groups, Baumslag--Solitar groups, Thompson's group \( F \), non-exceptional spherical Artin groups, and every torsion-free subgroup of a mapping class group of a finite-type 2-manifold (see \cite[Section~2.4]{VlamisHomeomorphism} for references). 

From the definitions, we see that every abstract group homomorphism from a completely metrizable topological group with the Rokhlin property to a cm-slender group is trivial.  
This statement can be extended to completely metrizable groups with the virtual Rokhlin property:

\begin{Prop}
\label{prop:trivial}
Every homomorphism from a completely metrizable topological group with the virtual Rokhlin property to a cm-slender group is trivial.
\end{Prop}

\begin{proof}
Let \( G \) be a completely metrizable topological group with the virtual Rokhlin property, let \( H \) be a cm-slender group, and let \( \vp \co G \to H \) be an abstract group homomorphism. 
Now let \( N < G \) be an open finite-index subgroup with the Rokhlin property.
Observe that a group with the Rokhlin property cannot have a proper open normal subgroup, as the complement of such a subgroup is also open and conjugation invariant.
Therefore, the intersection of all the conjugates of \( N \)---of which there are finitely many---is an open normal subgroup of \( N \) and hence equal to \( N \); in particular, \( N \) is a normal subgroup of \( G \). 
Now,  as \( \vp \) restricted to \( N \) must be trivial, we have that \( \vp \) factors through \( G/N \) and the image of \( \vp \) is finite.
It is well known that every cm-slender group is torsion free (this can be deduced from \cite[Example~1.4]{RosendalAutomatic}), and hence \( \vp \) must be trivial.
\end{proof}

Combining Theorem~\ref{thm:main1}, Proposition~\ref{prop:trivial}, and the fact that homeomorphism groups and mapping class groups are completely metrizable topological groups, we have the following corollary:

\begin{Cor}
\label{cor:homomorphism}
Let \( M \) be a 2-manifold.
If \( M \) is spacious, then every homomorphism from either \( \Homeo(M) \) or \( \mcg(M) \) to a cm-slender group is trivial.
\qed
\end{Cor}

\subsection*{Second countable Stone spaces}

A \emph{Stone space} is a compact zero-dimensional Hausdorff topological space. 
The end space of any manifold is a second-countable Stone space, and every second-countable Stone space can be realized as the end space of a planar 2-manifold. 
The homeomorphism group of a manifold \( M \) acts on the end space of \( M \) by homeomorphisms yielding a canonical homomorphism \( \Homeo(M) \to \Homeo(\Ends(M)) \). 
It follows from Richards's classification of surfaces \cite{Richards} that if \( M \) is a planar 2-manifold, then the canonical homomorphism \( \Homeo(M) \to \Homeo(\Ends(M)) \) is surjective. 
The discussion above yields the following corollary. 

\begin{Cor}
If \( E \) is a Stone space that can be realized as the end space of a spacious planar 2-manifold, then \( \Homeo(E) \) has the virtual Rokhlin property. 
\end{Cor}

Note that the converse of the corollary is false: the homeomorphism group of the Cantor set has the Rokhlin property \cite{GW}, but it is not the end space of a spacious 2-manifold (see Lemma~\ref{lem:self-similar_spacious}).

\subsection*{Acknowledgements}

The authors thank Kathryn Mann for helpful discussions and the referee for their careful reading and helpful comments. 
The first author is supported by NSF Grant DMS--2002187.
The second author is supported by NSF Grant DMS--2212922 and PSC-CUNY Award \#~64129-00 52.


\section{Preliminaries}
\label{sec:preliminaries}

We give here a small collection of preliminaries; for further details and background, we refer the reader to Section~2 of our article \cite{LanierMapping}, to which this article is a sequel.
We slightly break with the notation in \cite{LanierMapping} in that we assume 2-manifolds to be connected. 
We will remind the reader of essential background items from \cite{LanierMapping} as they appear in the article. 
Here, we briefly introduce several definitions and notations that will be used in several sections.

The \emph{compact-open topology} on the homeomorphism group \( \Homeo(M) \) of a 2-manifold \( M \) is generated by the sets of the form \( U(K,W) = \{ f \in \Homeo(M) : f(K) \subset W \} \), where \( K \subset M \) is compact and \( W\subset M \) is open and precompact\footnote{The precompact condition is not standard, but this additional requirement does not change the topology generated, as \( K \) is compact and \( M \) is locally compact.}.

A topological group \( G \) has the \emph{joint embedding property}, or \emph{JEP}, if for any two open subsets \( U \) and \( V \) in \( G \), there exists \( g \in G \) such that \( U \cap V^g \neq \varnothing \).
Here, \( V^g = \{ gvg^{-1}: v \in V \} \), and more generally, we will write \( v^g = gvg^{-1} \) for \( v,g \in G \). 
The following proposition is well known (see \cite[Theorem~2.2]{LanierMapping} for a proof):

\begin{Prop}
A Polish group has the Rokhlin property if and only if it has the joint embedding property.
\end{Prop}

Before turning to discussing 2-manifolds, we record a fact that comes out of the proof of Proposition~\ref{prop:trivial} that is useful in its own right.

\begin{Prop}
\label{prop:vr-normal}
Let \( G \) be a topological group with the virtual Rokhlin property. 
If \( N  \) is a closed finite-index subgroup of \( G \) with the Rokhlin property, then 
\begin{enumerate}[(i)]
\item \( N \) has no proper open normal subgroups.
\item \( N \) is normal in \( G \).
\item \( N \) is the unique closed finite-index subgroup of \( G \) with the Rokhlin property.
\end{enumerate}
\end{Prop}

\begin{proof}
Both (i) and (ii) were proved in the process of proving Proposition~\ref{prop:trivial}.
For (iii), if \( N' \) is another closed finite-index subgroup of \( G \) with the Rokhlin property, then \( N' \cap N \) is an open normal subgroup of both \( N \) and \( N' \), and hence by (i), \( N' \cap N = N \) and \( N' \cap N = N' \); in particular, \( N = N' \). 
\end{proof}

Next, let us record two facts about weakly self-similar 2-manifolds.
The first follows from a collection of results in \cite{MannRafi} (and is summarized in \cite[Theorem~2.4]{LanierMapping}).

\begin{Prop}
\label{prop:ss-partition}
If \( M \) is a weakly self-similar 2-manifold, then its set of maximally stable ends is either a singleton, doubleton, or is perfect.
Moreover, a weakly self-similar 2-manifold is self-similar if and only if its set of maximally stable ends is either a singleton or is perfect. 
\end{Prop}

The second fact  was hinted at in the introduction and is a consequence of the classification of surfaces and \cite[Theorem~2.4 and Lemma~5.4]{LanierMapping}.

\begin{Prop}
\label{prop:uniform}
Let \( M \) be a weakly self-similar 2-manifold with exactly two maximally stable ends.
Then, there exists a uniquely self-similar 2-manifold \( N \) such that \( M \) is homeomorphic to \( N \# N \).
Moreover, given any simple closed curve \( c \) in \( M \) separating the maximally stable ends of \( M \), the closure of each of the connected components of \( M \ssm c \) is homeomorphic to \( N \) with an open disk removed. 
\qed
\end{Prop}

Finally, we note that if \( M \) is orientable (resp., non-orientable), then  \( \Homeo_0(M) \) is a closed subgroup of \( \Homeo^+(M) \) (resp., \( \Homeo(M) \)), and hence the quotient homomorphism from \( \Homeo^+(M) \) (resp., \( \Homeo(M) \)) to \( \mcg(M) \)  is open and continuous.
In particular, if \( \Homeo(M) \) has the virtual Rokhlin property, then so does \( \mcg(M) \).


\section{Reducing to the case of two maximally stable ends}

To prove our main theorem, we partition 2-manifolds into four categories, as we did in \cite[Theorem~2.4]{LanierMapping} (with updated nomenclature).
Recall that a self-similar 2-manifold is \emph{uniquely self-similar} if it has a unique maximally stable end and it is \emph{perfectly self-similar} if its set of maximally stable ends is perfect\footnote{Recall that the empty set is perfect; in particular, the 2-sphere is perfectly self-similar. Malestien--Tao  in \cite{MalesteinSelf} use the terminology \emph{uniformly self-similar} instead of perfectly self-similar.}.
Using Proposition~\ref{prop:ss-partition}, we can now partition the class of 2-manifolds into four categories as follows:
\begin{enumerate}
\item uniquely self-similar 2-manifolds,
\item perfectly self-similar 2-manifolds,
\item weakly self-similar 2-manifolds with two maximally stable ends (and hence those that are not self-similar), and
\item non-weakly self-similar 2-manifolds.
\end{enumerate}

The goal of this section is to reduce the proof of Theorem~\ref{thm:main1} to the study of weakly self-similar 2-manifolds with exactly two maximally stable ends, i.e., in this section we will handle the proof of Theorem~\ref{thm:main1}  for  cases (1), (2), and (4)  above.
To do this, we rely heavily on the authors' previous work \cite{LanierMapping}.  
We note that the results in \cite{HHMRSV, LanierMapping} assume the 2-manifolds under consideration are orientable and are about mapping class groups rather than homeomorphism groups. 
However, Estrada \cite{EstradaConjugacy} extended the results to the non-orientable setting (see also \cite{VlamisHomeomorphism}) and the second author \cite[Theorem~11.1]{VlamisHomeomorphism} has extended the results to the setting of homeomorphism groups. 
Together these results yield:

\begin{Thm}
\label{thm:import}
Let \( M \) be a 2-manifold with infinite mapping class group. 
If \( M \) is orientable, let \( H = \Homeo^+(M) \); otherwise, let \( H = \Homeo(M) \).
Then \( H \) (resp., \( \mcg(M) \)) has the Rokhlin property if and only if \( M \) is uniquely self-similar.  
\qed
\end{Thm}

The infinite mapping class group assumption was added as previous work does not handle the case  of the projective plane, which we will deal with later in this section.

\begin{Lem}
\label{lem:self-similar_spacious}
A non-compact self-similar 2-manifold is spacious if and only if it is uniquely self-similar. 
\end{Lem}

\begin{proof}
Let \( M \) be a non-compact self-similar 2-manifold. 
First assume that \( M \) is perfectly self-similar. 
Fix a separating simple closed curve \( a_1 \) on \( M \) such that each component of \( M \ssm a_1 \) is a neighborhood of a maximally stable end of \( M \). 
Given \( n \in \bn \) with \( n > 1 \), there exist separating simple closed curves \( a_2, \ldots, a_n \) such that \( a_1, \ldots, a_n \) are pairwise-disjoint and such that each component of \( M \ssm \bigcup_{i=1}^n a_i \) is a neighborhood of a maximally stable end of \( M \) (this can be deduced from \cite[Corollary~4.6]{VlamisHomeomorphism}). 
The fact that \( a_j \in \Homeo(M) \cdot a_1 \) follows from the change of coordinates principle.
Observe that, for \( i \neq j \), any self-homeomorphism of \( M \) mapping \( a_i \) onto \( a_j \) must induce a nontrivial permutation of the maximally stable ends of \( M \); in particular, such a homeomorphism is not supported outside a neighborhood of the maximally stable ends.
Hence, \( | \Gamma(M) \backslash \Homeo(M)\cdot a_1 | \geq n \), and as \( n \) was arbitrary, \( M \) cannot be spacious. 

Now suppose \( M \) is uniquely self-similar. 
If \( M \) is orientable, let \( H = \Homeo^+(M) \) if \( M \); otherwise, let \( H = \Homeo(M) \).
By \cite[Proposition~4.2]{LanierMapping} (see also \cite[Lemma~11.5]{VlamisHomeomorphism}), \( \Gamma(M) \) is dense in \( H \).
If \( a \) and \( b \) are simple closed curves in \( M \) in the same \( H \)-orbit, the subset of \( H \) consisting of homeomorphisms mapping \( a \) onto a curve homotopic to \( b \) is open.
Hence, by the density of \( \Ht(M) \), there exists \( \gamma \in \Ht(M) \) such that \( \gamma(a) \) is homotopic to \( b \).
There is a compactly supported ambient isotopy from \( \gamma(a) \) to \( b \), and post-composing \( \gamma \) with the final step of the isotopy yields an element of \( \Ht(M) \) mapping \( a \) onto \( b \). 
Therefore, \( M \) is spacious.
\end{proof}

Before the next lemma, we introduce the following theorem from \cite{VlamisHomeomorphism}.

\begin{Thm}[{\cite[Corollary~2.4]{VlamisHomeomorphism}}]
\label{thm:finite-index}
Let \( M \) be a perfectly self-similar 2-manifold.
If \( M \) is orientable (resp., non-orientable), then \( \Homeo^+(M) \) (resp., \( \Homeo(M) \)) has no proper finite-index subgroups. 
\qed
\end{Thm}

\begin{Lem}
\label{lem:ss-case}
Let \( M \) be a self-similar 2-manifold. 
Then, \( \Homeo(M) \)  (resp., \( \mcg(M) \)) has the virtual Rokhlin property if and only if \( M \) is spacious. 
\end{Lem}

\begin{proof}
The homeomorphism group of the 2-sphere has the virtual Rokhlin property \cite{GW} (see also \cite[Theorem~11.4]{VlamisHomeomorphism}), and the 2-sphere is spacious.
We have already noted that the only compact self-similar 2-manifold is the 2-sphere, so for the remainder of the proof, we may assume that \( M \) is a non-compact self-similar 2-manifold.

First assume that \( M \) is not spacious, and hence perfectly self-similar by Lemma~\ref{lem:self-similar_spacious}.
By Theorem~\ref{thm:finite-index}, if \( M \) is perfectly self-similar, then \( \mcg(M) \) has the virtual Rokhlin property if and only if it has the Rokhlin property; but, \( \mcg(M) \) does not have the Rokhlin property (\cite[Lemma~4.8]{LanierMapping}).
Therefore, neither \( \Homeo(M) \) nor \( \mcg(M) \) have the virtual Rokhlin property.

Now suppose that \( M \) is spacious, and hence, by Lemma~\ref{lem:self-similar_spacious}, has a unique maximally stable end. 
By Theorem~\ref{thm:import}, \( \Homeo(M) \) and \( \mcg(M) \) have the virtual Rokhlin property. 
\end{proof}  

In \cite[Theorem~3.1]{LanierMapping}, the authors showed that the existence of a non-displaceable compact subset in an orientable 2-manifold \( M \) is an obstruction to \( \mcg(M) \) having the Rokhlin property whenever \( \mcg(M) \) is infinite.
The proof does not rely on the orientability of \( M \), and with only minor modifications (e.g., replacing elements with powers that lie in a given finite-index subgroup), the proof also yields an obstruction for the virtual Rokhlin property, which we record here:

\begin{Lem}
\label{lem:displaceable}
Let \( M \) be a 2-manifold with infinite mapping class group.
If \( M \) contains a non-displaceable compact subsurface, then neither \( \Homeo(M) \) nor  \( \mcg(M) \) have the virtual Rokhlin property. 
\qed
\end{Lem}

To finish, we need to consider several sporadic cases: the non-spacious 2-manifolds with finite mapping class groups (the thrice-punctured sphere, the projective plane, the Klein bottle, the open M\"obius band, and the twice-punctured projective plane).
Most of the work is embedded in the following two lemmas, one for non-orientable 2-manifolds and the other for orientable 2-manifolds. 
These two lemmas also establish a substantial part of Theorem~\ref{thm:main3}.

\begin{Lem}
\label{lem:homeo0_no}
If \( M \) is a finite-type non-orientable 2-manifold, then \( \Homeo_0(M) \) does not have the Rokhlin property. 
\end{Lem}

\begin{proof}
Fix three pairwise-disjoint embedded closed disks \( D_1 \), \( D_2 \), and \( D_3 \) in \( M \), and let \( D_k^\mathrm{o} \) denote the interior of \( D_k \). 
Fix an orientation on each of the \( D_k \), and let \( U^+(D_k, D_k^\mathrm{o}) \) (resp., \( U^- (D_k, D_k^\mathrm{o}) \)) denote the subset of \( U(D_k,D_k^\mathrm{o}) \) that preserves (resp., reverses) the orientation of \( D_k \).
Observe that both \( U^+(D_k, D_k^\mathrm{o}) \) and \( U^-(D_k, D_k^\mathrm{o}) \) are open subsets of \( \Homeo(M) \) and their definition does not depend on the choice of orientation on \( D_k \). 

For each \( i \in \{1,2,3,4\} \), fix oriented simple closed curves \( c_{i,1}, \ldots, c_{i,m} \) on \( M \) such that their homology classes generate \( H_1(M; \bz) \)  and such that \( c_{i,j} \) and \( c_{k,j} \) are disjoint and isotopic for \( i \neq k \). 
For each \( c_{i,j} \), fix an open regular neighborhood \( A_{i,j} \) of \( c_{i,j} \) such that \( A_{k,l} \cap A_{i,j} = \varnothing \) whenever \( k \neq i \) or \( l \neq j \). 
Similar to the above, let \( U^+(c_{i,j}, A_{i,j}) \) denote the (open) subset of \( U(c_{i,j}, A_{i,j}) \) consisting of elements that preserve the orientation of \( c_{i,j} \).

Let
\[ 
U = \Homeo_0(M)  \cap U^-(D_1, D_1^\mathrm{o}) \cap U^+(D_2, D_2^\mathrm{o}) \cap U^+(D_3, D_3^\mathrm{o}) 
\]
and
\[
V = \Homeo_0(M) \cap \bigcap_{i=1}^4\bigcap_{j=1}^m U^+(c_{i,j}, A_{i,j} ) .
\]
In words, \( U \) is the open subset of \( \Homeo_0(M) \) consisting of homeomorphisms that map \( D_k \) into its interior such that the images of \( D_2 \) and \( D_3 \) have the same orientation as \( D_2 \) and \( D_3 \), respectively, while the image of \( D_1 \) has the opposite orientation as \( D_1 \). 
On the other hand, \( V \) consists of the homeomorphisms of \( \Homeo_0(M) \) that send each \( c_{i,j} \) to a curve that is isotopic to itself within \( A_{i,j} \) (and has the same orientation). 

We claim that \( U \cap V = \varnothing \).
Suppose not and fix \( f \in U \cap V \).
By the Brouwer fixed point theorem, there exists \( x_k \in D_k^\mathrm{o} \) such that \( f(x_k) = x_k \) for \( k \in \{1,2,3\} \).
For each \( k \in\{1,2,3\} \) and for each \( j \in\{1, \ldots, m\} \), there exists \( i \in\{1,2,3,4\} \) such that \( x_k \not\in A_{i,j} \).
Up to relabeling, we may assume that \( A_{1,j} \) is disjoint from \( \{x_1, x_2, x_3 \} \) for all \( j \in \{1, \ldots, m\} \). 

Let \( M' = M \ssm \{x_1,x_2,x_3\} \), and view \( f \in \Homeo(M') \). 
Let \( d_2 = \partial D_2 \) and \( d_3 = \partial D_3 \). 
Observe that the homology classes of \( c_{1,1}, \ldots, c_{1,m}, d_2, d_3 \) generate \( H_1(M', \bz) \).
By assumption, \( f \) fixes the homology classes of each of these curves and hence acts trivially on \( H_1(M', \bz) \).
However, this contradicts the fact that \( f \) must negate the homology class of \( \partial D_1 \); hence, \( U \cap V = \varnothing \). 

To finish, given \( \sigma \in \Homeo_0(M) \), the above argument can be applied with \( c_{i,j} \) replaced by \( \sigma(c_{i,j}) \) and \( A_{i,j} \) replaced by \( \sigma(A_{i,j}) \) to see that \( U \cap V^\sigma = \varnothing \). 
Therefore, \( \Homeo_0(M) \) fails to have the JEP, and hence fails to have the Rokhlin property.
\end{proof}

\begin{Lem}
\label{lem:homeo0_o}
Let \( M \)  be a finite-type orientable 2-manifold that is not homeomorphic to the 2-sphere, plane, or open annulus. 
Then, \( \Homeo_0(M) \) does not have the Rokhlin property.
\end{Lem}

\begin{proof}
By the classification of surfaces and the hypotheses on \(M\), either \(M\) has essential simple closed curves or else it is homeomorphic to the thrice-punctured sphere. 
Fix an embedded closed disk \( D \) in \( M \), fix a point \( p \in D^\mathrm{o} \), and let \( M_p = M \ssm \{p\} \). 
Let \( f \in \Homeo_0(M) \) be such that \( f(p) = p \) and such that given any simple closed curve (resp., simple proper arc) \( c \) in \( M_p \), every simple closed curve (resp., simple proper arc) isotopic to \( f(c) \) in \( M_p \) intersects \( c \) nontrivially.
This can be accomplished by taking a high power of a homeomorphism in \( \Homeo_0(M) \) fixing \( p \) that is pseudo-Anosov as a homeomorphism of \( M_p \) and whose isotopy class is in the kernel of the forgetful homomorphism \( \mcg(M_p) \to \mcg(M) \). 

Choose an Alexander system \( \{c_1, \ldots, c_n\} \) for \( M_p \) and closed annular neighborhoods \( A_1, \ldots, A_n \) of \( c_1, \ldots, c_n \), respectively, such that \( A_i \) is disjoint from \( D \) for each \( j \in \{1, \ldots, n \} \),  and let 
\[
U = \bigcap_{i=1}^n U(c_i, f(A_i^\mathrm{o})). 
\]  
Note that, on account of the \( c_i \) being an Alexander system, every element of \( U \) fixing \( p \) is isotopic to \( f \) as a homeomorphism of \( M_p \).
If \( M \) is homeomorphic to a thrice-punctured sphere, let \( b_1 \) and \( b_2 \) be disjoint  properly embedded arcs in \( M \); otherwise, let \( b_1 \) and \( b_2 \) be disjoint essential simple closed curves in \( M \) (note:~they can be isotopic).
In the latter case, let \( B_1 \) and \( B_2 \) be disjoint closed annular neighborhoods of \( b_1 \) and \( b_2 \), respectively, and in the former case, let \( B_1 \) and \( B_2 \) be disjoint closed neighborhoods of \( b_1 \) and \( b_2 \), respectively, such that \( B_j \) is homeomorphic to \( b_j \times [-1,1] \). 
Let \( V = U(D,D^\mathrm{o}) \cap U(b_1, B_1^\mathrm{o}) \cap U(b_2, B_2^\mathrm{o}) \). 

We claim that \( U \cap V  = \varnothing \): Suppose not and let \( g \in U \cap V \). 
As \( g \in   V \) and hence in \( U(D, D^\mathrm o) \), there exists some \( q \in D^\mathrm{o} \) such that \( g(q) = q \). 
At least one of \( B_1 \) and \( B_2 \) is contained in \( M \ssm\{q\} \); let us assume it is \( B_1 \). 
Fix a homeomorphism \( \vp \co M \to M \) supported in \( D \) such that \( \vp(p) = q \). 
Then, we can view \( \vp \) as a homeomorphism from \( M_p \) to  \( M \ssm \{q\} \).
Observe that \( \vp\circ f \circ \vp^{-1}\) is isotopic to \( g \) in \( M \ssm\{q\} \), as the two homeomorphisms agree outside of  the disk \( D \) and every homeomorphism of \( D \) fixing \( q \) is isotopic to the identity via an isotopy fixing \( q \) at each stage.
In particular, by our choice of \( f \), we have \( g(b_1) \) cannot be isotoped off of \( b_1 \) in \( M_q \). 
However, as \( g \in V \), \( g(b_1) \) must be isotopic to a component of \( \partial B_1 \), each of which is disjoint from \( b_1 \), a contradiction. 
To finish, using the fact that \( b_i \) and \( B_i \) were arbitrary, the argument above establishes the fact that \( U \cap V^\sigma = \varnothing \) for every \( \sigma \in \Homeo_0(M) \), and hence \( \Homeo_0(M) \) does not have the JEP.
Therefore, \( \Homeo_0(M) \) does not have the Rokhlin property. 
\end{proof}

For a compact orientable 2-manifold \( M \) of genus at least one, Lemma~\ref{lem:homeo0_o} can be deduced using the work of Bowden--Hensel--Mann--Militon--Webb \cite{BowdenRotation} in which they show that the asymptotic translation length of the action of \( H_0(M) \) on the fine curve graph of \( M \) yields a continuous (conjugation-invariant) surjection onto the non-negative reals. 
Such a function cannot exist for a topological group with the Rokhlin property.

To finish the goal of this subsection, we require Theorem~\ref{thm:main3}, but in order to provide a proof, we need to introduce a lemma. 

\begin{Lem}
\label{lem:no-normal}
If \( M \) is a finite-type 2-manifold, then \( \Homeo_0(M) \) has no proper finite-index normal subgroups.
\end{Lem}

\begin{proof}
If \( M \) is closed, then \( \Homeo_0(M) \) is a simple group (\cite{FisherGroup}, see also \cite[Theorem~6.11]{VlamisHomeomorphism}); otherwise, the structure of the lattice of normal subgroups of \( \Homeo_0(M) \) is given in \cite[Theorem~1.1]{McDuffLattice}, and in particular, there are no proper finite-index normal subgroups.
\end{proof}

\begin{proof}[Proof of Theorem~\ref{thm:main3}]
Let \( M \) be a finite-type 2-manifold. 
By Lemma~\ref{lem:no-normal}, \( \Homeo_0(M) \) has the virtual Rokhlin property if and only if it has the Rokhlin property. 
If \( M \) is is homeomorphic to neither the 2-sphere, the plane, nor the annulus, then by Lemma~\ref{lem:homeo0_no} and Lemma~\ref{lem:homeo0_o}, \( \Homeo_0(M) \) does not have the Rokhlin property.
To finish, we note that both \( \Homeo_0(\mathbb S^2) \) and \( \Homeo_0(\mathbb R^2) \)  have the Rokhlin property (see \cite[Theorems~11.4~\&~11.6]{VlamisHomeomorphism}), and by Theorem~\ref{thm:2} in the next section, \( \Homeo_0(\mathbb S^1 \times \br) \) has the Rokhlin property. 
\end{proof}

The next lemma, together with Lemma~\ref{lem:ss-case}, accomplishes the goal laid out in the beginning of the section, allowing us to focus for the rest of the article on weakly self-similar 2-manifolds with exactly two maximally stable ends.  

\begin{Lem}
\label{lem:cd-case}
If \( M \) is a 2-manifold containing a non-displaceable proper compact subset, then \( \Homeo(M) \) does not have the virtual Rokhlin property. 
\end{Lem}

\begin{proof}
The case in which \( M \) has infinite mapping class group is handled by Lemma~\ref{lem:displaceable}. 
So let us assume that the mapping class group of \( M \) is finite. 
Note that this also implies that \( M \) is of finite type. 
Under these assumptions, \( \Homeo_0(M) \) is an open finite-index normal subgroup of \( \Homeo(M) \); by Proposition~\ref{prop:vr-normal}, it follows that \( \Homeo(M) \) has the virtual Rokhlin property if and only if \( \Homeo_0(M) \) has the virtual Rokhlin property. 
The result follows by applying Theorem~\ref{thm:main3}. 
\end{proof}


\section{Two maximally stable ends}

In the previous section, we saw that if \( \Homeo(M) \) has the virtual Rokhlin property, then \( M \) is weakly self-similar (Lemma~\ref{lem:cd-case}). 
And, we also saw that if \( M \) is self-similar, then \( \Homeo(M) \) has the virtual Rokhlin property if and only if \( M \) is spacious (Lemma~\ref{lem:ss-case}). 
Therefore, to establish Theorem~\ref{thm:main1}, it is left to prove the following: 

\begin{Thm}
\label{thm:2}
Let \( M \) be a weakly self-similar 2-manifold with exactly two maximally stable ends.
Then, \( \Homeo(M) \) (resp., \( \mcg(M) \)) has the virtual Rokhlin property if and only if \( M \) is spacious. 
\end{Thm}

Before proving Theorem~\ref{thm:2}, we first establish basic facts about how \( \Ht(M) \) sits inside of \( \Homeo(M) \) and how it acts on curves and subsurfaces.

\begin{Lem}
\label{lem:action}
Let \( a \) be a simple closed curve in \( M \), let \( \Ht = \Ht(M) \), and let \( X = \Gamma \backslash (\Homeo(M) \cdot a) \).
Then, \( \Homeo(M) \) acts on \( X \) via \( f \cdot (\Gamma \cdot h(a)) = \Gamma \cdot f(h(a)) \).  
Moreover, if \( X \) is spacious, then the kernel of this action is an open finite-index subgroup of \( \Homeo(M) \). 
\end{Lem}

\begin{proof}
The existence of the action follows readily from the fact that \( \Ht \) is normal. 
Let \( N \) denote the kernel of the action of \( \Homeo(M) \) on \( X \). 
Let \( x \in X \), and consider the set \[ U_x = \{ f \in \Homeo(M) : f\cdot x=x \}. \]
Note that \( N = \bigcap_{x \in X} U_x \).
Fix a representative simple closed curve \( c \) in \( x \), and let \( A \) be an open annular neighborhood of \( c \). 
Then, the open subset \( U(c, A) \) is contained in \( U_x \), and as \( U_x \) is a subgroup with nonempty interior, it is open. 
To finish, if \( M \) is spacious, then \( X \) is finite.
Hence, if \( M \) is spacious, then \( N \) is finite index and, as \( N \) is a finite intersection of open sets, is open.
\end{proof}

In what follows, it will be helpful to have the following notions:  in a weakly self-similar 2-manifold \( M \) with exactly two maximally stable ends, a \emph{tube} is a subsurface co-bounded by two disjoint simple closed curves, each of which separates the two maximally stable ends.
Observe that \( \Ht(M) \)  is the normal subgroup of \( \Homeo(M) \) consisting of elements whose support is contained in a tube. 
Given a tube \( T \) in \( M \), we say a homeomorphism \( \vp \co M \to M \) is a \emph{\( T \)-translation} if \( \vp^n(T) \cap \vp^m(T) = \varnothing \) for  any distinct \( n,m \in \bz \). 

\begin{Lem}
\label{lem:translation_exists}
Let \( M \) be a weakly self-similar 2-manifold with exactly two maximally stable ends. 
If \( T \) is a tube in \( M \), then a \( T \)-translation exists. 
\end{Lem}

\begin{proof}
This readily follows from Proposition~\ref{prop:uniform} and the change of coordinates principle. 
\end{proof}

\begin{Lem}
\label{lem:tube-disjoint}
Let \( M \) be a weakly self-similar 2-manifold with exactly two maximally stable ends. 
If \( M \) is spacious, then given any tube \( T \)  in \( M \) there exists \( T \)-translation in \( \Gamma(M) \). 
\end{Lem}

\begin{proof}
Choose a \( T \)-translation \( \vp \) in \( \Homeo(M) \). 
Up to replacing \( \vp \) with \( \vp^2 \), we may assume that \( \vp \) stabilizes the maximally stable ends of \( M \). 
Let \( a \) be the component of \( \partial T \) such that the tube bounded by \( a \) and \( \vp(a) \) contains \( T \). 
Then, as \( M \) is spacious, there exist distinct natural numbers \( i \) and \( j \) such that \( \vp^i(a) \) and \( \vp^j(a) \) are in the same \( \Gamma(M) \)-orbit.
Assume \( j > i \).
By Lemma~\ref{lem:action}, there exists \( g \in \Gamma(M) \) such that \( g(a) = \vp^{j-i}(a) \). 
Using that \( g \) stabilizes the maximally stable ends and that the tube co-bounded by \( a \) and \( g(a) \) contains \( T \), we must have that \( g(T) \cap T = \varnothing \) so that \( g \) is a \( T \)-translation. 
\end{proof}

\begin{figure}
 \labellist
 \small\hair 2pt

  \pinlabel {\(a_1\)} [ ] at 140 -15
    \pinlabel {\(b_1\)} [ ] at 300 -15
      \pinlabel {\(a_2\)} [ ] at 517 -15
        \pinlabel {\(b_2\)} [ ] at 677 -15
          \pinlabel {\(g_1(b_1)\)} [ ] at 740 -15
            \pinlabel {\(c\)} [ ] at 788 -15

  \pinlabel {\(T_1\)} [ ] at 220 30
    \pinlabel {\(T_2\)} [ ] at 600 30

 \endlabellist  
\includegraphics[width=\textwidth]{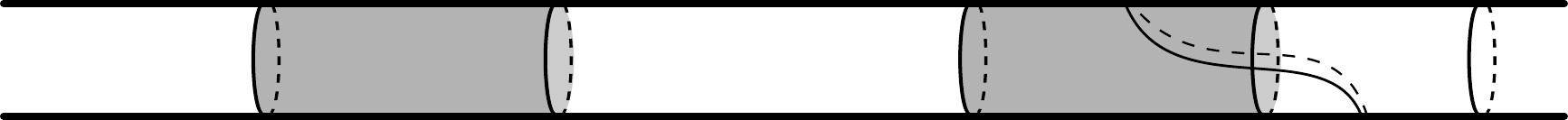}
\caption{A schematic of the curves and tubes that appear in the proof of Lemma~\ref{lem:ambient}.}
\label{fig:4.5}
\end{figure}
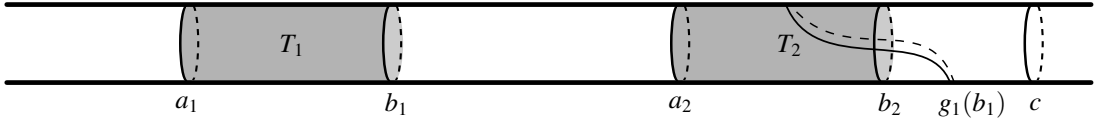

\begin{Lem}
\label{lem:ambient}
Let \( M \) be a spacious 2-manifold with exactly two maximally stable ends, and let \( \Ht = \Ht(M) \).
Let \( T_1 \) and \( T_2 \)  be disjoint homeomorphic tubes in \( M \), and let \( a_i \) and \( b_i \) denote the boundary components of \( T_i \), labelled such that both \( b_1 \) and \( a_2 \) separate \( a_1 \) from \( b_2 \).
If \( a_2 \in \Ht \cdot a_1 \) and \( b_2 \in \Ht \cdot b_1 \), then there exists \( g \in \Ht \) with \( g(T_1) = T_2 \). 
\end{Lem}

\begin{proof}
Let \( g_1 \in \Ht \) such that \( g_1(a_1) = a_2 \). 
Note that there is an element \( \gamma \) of \( \Ht \), supported in a tube \( T \),  mapping \( g_1(b_1) \) onto \( b_2 \). 
Let \( c \) denote the component of the boundary of \( T \) such that \( a_2 \) and \( c \) bound a tube containing \( g_1(b_1) \) and \( b_2 \). 
The tube bounded by \( g_1(b) \) and \( c \) is  homeomorphic to the tube bounded by \( b_2 = \gamma(g_1(b_1)) \) and \( c=\gamma(c) \); moreover, the tube bounded by \( a_2 \) and \( g_1(b_1) \) and the tube bounded by \( a_2 \) and \( b_2 \) are homeomorphic (as both are homeomorphic to \( T_1 \)).
Therefore, by the change of coordinates principle, there exists an ambient homeomorphism \( g_2 \) of \( M \) supported in the tube bounded by \( a_2 \) and \( c \) and mapping \( g_1(b_1) \) onto \( b_2 \).
Therefore, setting \( g = g_2\circ g_1 \in \Ht \)  we have \( g(T_1) =  T_2 \). 
\end{proof}

Over the next two lemmas, we turn to establishing the reverse direction of the statement in Theorem~\ref{thm:2}.
The proof of Lemma~\ref{lem:jep} is very similar to \cite[Lemma~4.3]{LanierMapping} (and \cite[Theorem~11.6]{VlamisHomeomorphism}), and so the novelty (in this direction) lies in the following lemma.

\begin{Lem}
\label{lem:dense}
Let \( M \) be a spacious 2-manifold with exactly two maximally stable ends, and let \( \Ht = \Ht(M) \).
Then \( \overline \Ht \), the closure of \( \Ht \) in \( \Homeo(M) \), is an open finite-index subgroup of \( M \).  
\end{Lem}

\begin{proof} 
Fix a separating simple closed curve \( c \) in \( M \) that separates the maximally stable ends of \( M \).
Let \( N \) be the subgroup of the kernel of the action of \( \Homeo(M) \) on \( \Gamma \backslash (\Homeo(M) \cdot c ) \) containing the elements that fix the maximally stable ends of \( M \) and that are orientation preserving if \( M \) is orientable.
By Lemma~\ref{lem:action}, \( N \) is an open finite-index subgroup of \( \Homeo(M) \).
We claim that \( \overline\Ht = N \). 

Fix a nonempty open subset \( U \) of \( N \). 
By shrinking \( U \), we may assume that there exist compact sets \( K_1, \ldots, K_m \) and precompact open sets \( V_1, \ldots, V_m \) such that \( U = \bigcap_{i=1}^m U(K_i, V_i) \). 
Fix \( f \in U \), and choose a tube \( T \) containing each of the \( K_i \) in its interior. 
Also, by enlarging \( T \) if necessary, we may assume that \( T \) is non-orientable if \( M \) is non-orientable.  
Let \( T_1 \) be a tube containing \( T \cup f(T) \).


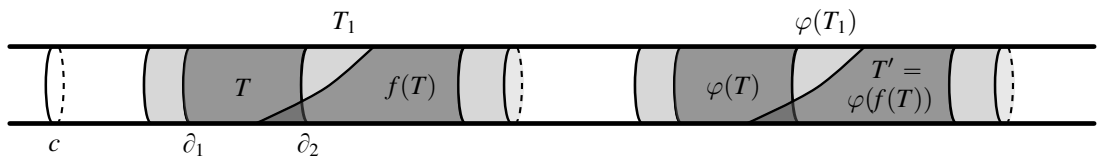
\begin{figure}
 \labellist
 \small\hair 2pt

  \pinlabel {\(\ \partial_1\)} [ ] at 140 -15
    \pinlabel {\( \partial_2\)} [ ] at 230 -15

 \pinlabel {\(c\)} [ ] at 36 -15

  \pinlabel {\(T_1\)} [ ] at 258 80
    \pinlabel {\(\vp(T_1)\)} [ ] at 626 80
    
  \pinlabel {\(T\)} [ ] at 180 30
    \pinlabel {\(f(T)\)} [ ] at 310 30
    
    \pinlabel {\(\vp(T)\)} [ ] at 555 30
        \pinlabel {\(T'=\)} [ ] at 680 45
        \pinlabel {\(\vp(f(T))\)} [ ] at 675 20

 \endlabellist  
\includegraphics[width=\textwidth]{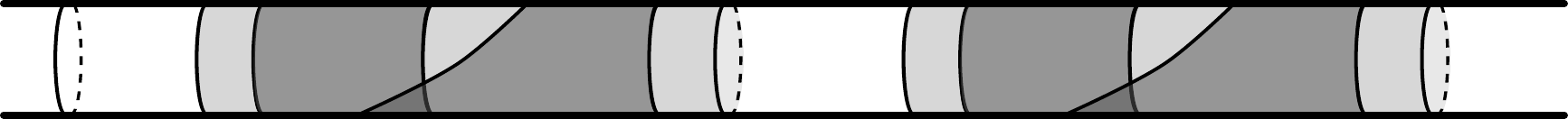}
\caption{A schematic of the curves and tubes that appear in the proof of Lemma~\ref{lem:dense}.}
\label{fig:4.6}
\end{figure}
 
 
By Lemma~\ref{lem:tube-disjoint}, there exists a \( T_1 \)-translation \( \vp \in \Gamma \).
Set \( T' = \vp(f(T)) \).
By construction,  \( T' \) is disjoint from \( T \cup f(T) \). 
Let \( \partial_1 \) and \( \partial_2 \) denote the boundary components of \( T \) (see Figure~\ref{fig:4.6} for a schematic).
By the definition of \( N \), the image \( f(\partial_i) \) is in the same \( \Gamma \)-orbit as \( \partial_i \); moreover, recall that by Proposition~\ref{prop:uniform}, \( \Homeo(M) \cdot c = \Homeo(M) \cdot \partial_i \).
Therefore,  by Lemma~\ref{lem:ambient}, there exists  \( g_1, g_2 \in \Ht \) such that \( g_1(T) = T' \) and \( g_2(T') = f(T) \).
Hence, \( (f^{-1} \circ g_2 \circ g_1) (T) = T \). 
In the case that \( M \) is non-orientable, it is possible that \( f^{-1} \circ g_2 \circ g_1 \) reverses the orientation of an annular neighborhood of one or each of the boundary components of \( M \).  
However, up to pre-composing \( g_1  \) with double slides (which are supported in tubes, see \cite[Section~5.1]{VlamisHomeomorphism}), we may assume that \( f^{-1} \circ g_2 \circ g_1 \) restricts to the identity on the boundary components of \( T \). 
We can therefore write \( f^{-1} \circ g_2 \circ g_1 = f_1 \circ f_2 \), where \( f_2 \) is supported on \( T \) and \( f_1 \) is supported in the closure of the complement of \( T \). 
It follows that \( f\circ f_1  = g_2\circ g_1 \circ f_2^{-1} \). 
Observe that each homeomorphism on the right-hand side is supported in a tube, and hence \( f \circ f_1 \in \Ht \). 
As \( f_1 \) leaves each of the \( K_i \) fixed, we have that \( f \circ f_1 (K_i) = f(K_i) \subset V_i \), and so \( f\circ f_1 \in U \) and \( \Ht \cap U \neq \varnothing \). 
Thus, as \( U \) was arbitrary, \( \Ht \) is dense in \( N \).
Therefore, \( \overline \Ht = N \) as \( N \) is a closed subgroup of \( \Homeo(M) \). 
\end{proof}

\begin{Lem}
\label{lem:jep}
If \( M \) is a spacious 2-manifold with exactly two maximally stable ends, then \( \Homeo(M) \) has the virtual Rokhlin property. 
\end{Lem}

\begin{proof}
By Lemma~\ref{lem:dense}, \( \overline \Ht \) is an open finite-index normal subgroup of \( \Homeo(M) \). 
Therefore it suffices to show that \( \overline \Ht \) has the JEP, and hence the Rokhlin property.
Given nonempty open subsets \( U_1 \) and \( U_2 \) of \( \overline \Ht \), we need to find \( \sigma \in \overline\Ht \) such that \( U_1 \cap U_2^\sigma \neq \varnothing \). 
By shrinking the \( U_i \), we may assume that there exists compact subsets \( K_1^i, \ldots, K_{n_i}^i \) and precompact open subsets \( V_1^i, \ldots, V_{n_i}^i \) such that \( U_i = \bigcap_{j=1}^{n_i} U(K_j^i, V_j^i) \).
For \( i \in \{1,2\} \), choose \( g_i \in \Ht \cap U_i \), and let \( T_i \) be a tube containing the support of \( g_i \) as well as  \( V_j^i \) and \( K_j^i \) for each \( j \in \{1, \ldots, n_i\} \). 

By applying Lemma~\ref{lem:tube-disjoint} to a tube containing \( T_1 \cup T_2 \), there exists \( \sigma \in \Ht \) such that \( \sigma(T_2) \) is disjoint from \( T_1 \cup T_2 \). 
Let \( \tau = g_1g_2^\sigma \in\Ht \). 
Then,  using the fact that \( g_1 \) restricts to the identity on \( \sigma(V_j^2) \) and that \( g_2^\sigma \) restricts to the identity on \( K_j^1 \), we have:
\begin{enumerate}[(1)]
\item for \( j \in \{1, \ldots, n_1\} \), \[ \tau(K_j^1) = g_1g_2^\sigma(K_j^1) = g_1(K_j^1) \subset V_j^1 \] implying \( \tau \in U_1 \), and
\item for \( j \in \{1, \ldots, n_2\} \),
\[
\tau(\sigma(K_j^2)) = (g_1g_2^\sigma)(\sigma(K_j^2)) = g_1(\sigma(g_2(K_j^2))) \subset g_1( \sigma(V_j^2))  = \sigma(V^2_j)
\]
implying \( \tau \in U_2^\sigma \). 
\end{enumerate}
Therefore, \( \tau \in U_1 \cap U_2^\sigma \), establishing that  \( \overline\Ht \) has the JEP, and hence \( \Homeo(M) \) has the virtual Rokhlin property.
\end{proof}

We now move to proving the forwards direction of Theorem~\ref{thm:2}. 
When assuming our group has the virtual Rokhlin property, we are handed a finite-index open normal subgroup that we a priori know relatively little about.
The next lemma will aid in moving from working with this unknown subgroup to working with \( \Ht \).

\begin{Lem}
\label{lem:G-telescoping}
Let \( M \) be a weakly self-similar 2-manifold with exactly two maximally stable ends, let \( G \) be an open normal finite-index subgroup of  \( \Homeo(M) \), and let \( \Ht = \Ht(M) \).
Suppose \( a \) and \( b \) are simple closed curves in the same \( G \)-orbit. 
If either \( a \) does not separate the maximally stable ends of \( M \) or 
there exists \( g \in G \) and a simple closed curve \( c \) in \( M \) such that
\begin{enumerate}[(i)]
\item \( c \) separates the maximally stable ends of \( M \),  
\item \( a\cup b \) is contained in a component of \( M \ssm c \), 
\item \( g(a) = b \), and 
\item \( g \) restricts to the identity on \( c \), 
\end{enumerate}
then \( a \) and \( b \) are in the same \( \Ht \)-orbit.
\end{Lem}

\begin{proof} 
Choose a tube \( T \) containing \( a \) and \( b \) in its interior. 
If \( a \) and \( b \) are non-separating, then there exists a homeomorphism supported in \( T \) mapping \( a \) onto \( b \) by the change of coordinates principle, and hence \( a \) and \( b \) are in the same \( \Gamma \)-orbit.
Now we may assume \( a \) and \( b \) are separating.
First, suppose \( a \) (and hence \( b \)) does not separate the maximally stable ends of \( M \).
By Lemma~\ref{lem:translation_exists}, there exists a \( T \)-translation in \( \Homeo(M) \), and taking a suitable power yields a \( T \)-translation  \( \tau \) in \( G \). 
Fix a tube \( T' \) containing \( T \) and \( \tau(T) \). 
Let \( F_a \) be the closure of the component of \( M \ssm a \) that is not a neighborhood of a maximally stable end of \( M \). 
Choose a simple path \( \alpha \) in \( T' \) from \( a \) to \( \tau(a) \), and let \( \Sigma \) be a regular neighborhood of \( F_a \cup \alpha \cup \tau(F_a) \). 
Then, by the change of coordinates principle, there exists a homeomorphism \( f_a \) of \( M \) supported in \( \Sigma \) mapping \( a \) onto \( \tau(a) \).
Similarly, there exists \( f_b \in \Gamma \) mapping \( b \) onto \( \tau(a) \). 
Therefore,  \( f_b^{-1} \circ f_a \in \Ht \) maps \( a \) onto \( b \). 

\begin{figure}
\centering
 \labellist
 \small\hair 2pt

  \pinlabel {\(\sigma(b)\)} [ ] at 120 -15
    \pinlabel {\(b\)} [ ] at 305 -15
      \pinlabel {\(a\)} [ ] at 211 -15
        \pinlabel {\(c\)} [ ] at 394 -15
          \pinlabel {\(\sigma(c)=d\)} [ ] at 28 -15

 \endlabellist  
\includegraphics[width=.7\textwidth]{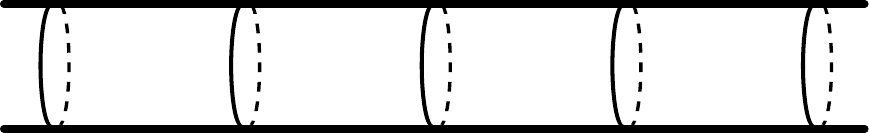}
\vskip.1in
\caption{The positioning of the curves in the proof of Lemma~\ref{lem:G-telescoping}.}
\label{fig:4.8}
\end{figure}
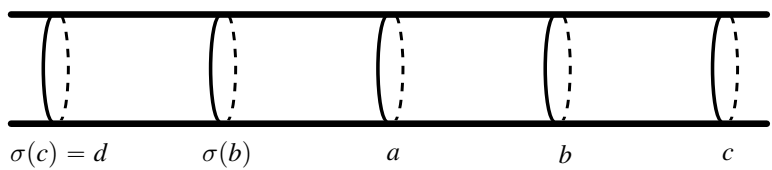

Now, we may assume that \( a \) (and hence \( b \)) separates the maximally stable ends of \( M \). 
For now, let us additionally assume that \( a \) and \( b \) are disjoint. 
Then, by our assumptions, there exists a simple closed curve \( c \) separating the maximally stable ends of \( M \), and a homeomorphism \( g \in G \) such that \( a \cup b \) is contained in a component of \( M \ssm c \), \( g(a) = b \), and \( g \) restricts to the identity on \( c \). 
By possibly relabelling \( a \) and \( b \) and replacing \( g \) with \( g^{-1} \), we may assume that \( b \) separates \( a \) and \( c \). 
Since the components of \( M \ssm a \) are homeomorphic, there exists a homeomorphism \( \sigma \co M \to M \) that permutes the components of \( M \ssm a \); let \( d = \sigma(c) \).
Let \( T_{x,y} \) be the tube co-bounded by \( x \) and \( y \) for distinct \( x,y \in \{a,b, c, d\} \). 
Then \( T_{a,c} \), \( T_{b,c} \), \( T_{d,a} \), and \( T_{d,b} \) are pairwise homeomorphic.  
Indeed, the restriction of \( \sigma \) to \( T_{a,c} \) is a homeomorphism between \( T_{a,c} \) and \( T_{d,a} \).
The restriction of \( g \) to \( T_{a,c} \) is a homeomorphism from \( T_{a,c} \) to \( T_{b,c} \).
The end space of \( T_{d,b} \) is the disjoint union of the end space of \( T_{d,a} \) and the end space of \( T_{a,b} \), which is homeomorphic to the disjoint union of the end spaces of \( T_{b,c} \) and \( T_{a,b} \), and hence homeomorphic to the end space of \( T_{a,c} \); therefore, by the classification of surfaces, \( T_{d,b} \) is homeomorphic to \( T_{a,c} \). 
Therefore, by the change of coordinates principle, there exists a homeomorphism \( M \to M \) supported in \( T_{d, c} \) mapping \( a \) onto \( b \). 

To finish, note that if \( a \cap b \neq \varnothing \), then we can simply choose a simple closed curve \( \gamma \) disjoint from both \( a \) and \( b \) and use the previous case to find homeomorphisms \( f_a, f_b\co M \to M \) supported in tubes such that \( f_a(a) = \gamma \) and \( f_b(b) = \gamma \), so that \( f_b^{-1} \circ f_a \co M \to M \) is a homeomorphism supported in a tube that maps \( a \) onto \( b \). 
\end{proof}

It is left for us to establish the forward direction of Theorem~\ref{thm:2}. 
The outline of the argument is as follows: we are handed a finite-index subgroup \( G \) of \( \Homeo(M) \) with the Rokhlin property. 
Given two simple closed curves \( a \) and \( b \) in the same \( G \)-orbit, we can take the open set \( U \) in \( \Homeo(M) \) consisting of elements that map \( a \) to a curve homotopic to \( b \).
We can then choose (with some care) another open subset \( V \) in \( \Homeo(M) \) consisting of elements that fix the homotopy class of some simple closed curve \( c \) that separates the maximally stable ends of \( M \) and is disjoint from, but does not separate, \( a \) and \( b \). 
Then, by invoking the Rokhlin property, there exists an element \( \sigma \in G \) such that \( U \cap V^\sigma \) is nonempty. 
An element in this intersection must---up to homotopy---fix \( \sigma(c) \) and send \( a \) to \( b \). 
If \( \sigma(c) \) is disjoint from \( a \) and \( b \) and does not separate \( a \) and \( b \), then we can apply Lemma~\ref{lem:G-telescoping} to see that \( a \) and \( b \) are in the same \( \Gamma \)-orbit.
Therefore, the main hurdle we  have to deal with below is showing that \( \sigma(c) \) satisfies this condition. 
To do so, we make use of a pseudo-Anosov homeomorphism supported on a subsurface, its action on the curve graph, and subsurface projections. 
For the relevant background, we refer the reader to \cite[Sections~2.3.2~\&~2.3.3]{LanierMapping}; we note that \( \cc(S) \) denotes the curve graph of the surface \( S \). 

\begin{proof}[Proof of Theorem~\ref{thm:2}] 
We established the reverse direction in Lemma~\ref{lem:jep}.
For the forwards direction, assume that \( \mcg(M) \) has the virtual Rokhlin property, which is weaker than assuming \( \Homeo(M) \) has the virtual Rokhlin property. 
First note that if \( \mcg(M) \) is finite, then \( M \) is homeomorphic to \( \mathbb S^1 \times \br \) and hence spacious. 
So, we may assume that \( \mcg(M) \) is infinite. 
Let \( G \) be an open finite-index subgroup of \( \mcg(M) \) with the Rokhlin property; let \( n \) be the index of \( G \) in \( \mcg(M) \).  
As we saw in the proof of Proposition~\ref{prop:trivial}, \( G \) is normal in \( \mcg(M) \). 
If \( M \) is orientable (resp., non-orientable), let \( \widetilde G \) be the preimage of \( G \) under the canonical projection from  \( \Homeo^+(M) \) (resp., \( \Homeo(M) \)) to \( \mcg(M) \). 

As \( \widetilde G \) is finite index in \( \Homeo(M) \), we have that \( \widetilde G \backslash ( \Homeo(M)\cdot d) \) is finite for any simple closed curve \( d \) in \( M \). 
Therefore, showing the \( \widetilde G \)-orbit and \( \Ht \)-orbit agree for any simple closed curve on \( M \) will establish that \( M \) is spacious. 
Fix two simple closed curves \( a \) and \( b \) in \( M \) that are in the same \( \widetilde G \)-orbit. 
If \( a \) (and hence \( b \)) does not separate the maximally stable ends, by Lemma~\ref{lem:G-telescoping}, there is an ambient homeomorphism of \( M \) supported in a tube mapping \( a \) onto \( b \). 
So, let us assume that \( a \) (and hence \( b \)) separate the maximally stable ends of \( M \). 
For the time being, let us further assume that \( a \) and \( b \) are disjoint. 

Let \( \Sigma \) be a subsurface of \( M \) satisfying the following properties:
\begin{itemize}
\item \( \Sigma \) is connected, compact, and orientable,
\item each boundary component of \( \Sigma \) is separating and in minimal position with \( a \), 
\item \( \Sigma \) separates the maximally stable ends of \( M \), and
\item \( \Sigma \) admits a pseudo-Anosov homeomorphism (for instance, it is enough to require  \( \Sigma \) to have at least four boundary components). 
\end{itemize}
Let \( f \) be a pseudo-Anosov homeomorphism of \( \Sigma \) fixing  \( \partial \Sigma \) point-wise and such that  \( d_{\mathcal C(\Sigma)}(f\cdot \eta, \eta) > 2 \) for any \( \eta \in \mathcal C(\Sigma) \) (this can be accomplished by taking a suitably high power of any pseudo-Anosov homeomorphism of \( \Sigma \)). 
Extending \( f \) by the identity to the complement of \( \Sigma \) in \( M \), we will view \( f \in \Homeo(M) \). 

Let  \( \alpha \) and \( \beta \) denote the homotopy classes of \( a \) and \( b \), respectively.
The set \[ U_{\alpha,\beta} = \{ \vp \in G : \vp(\alpha) = \beta \} \] and, viewing \( \mathcal C(\Sigma) \) as a subset of \( \mathcal C(M) \), the set
\[ U_\Sigma = \{ \vp \in G : \vp(\eta) = \eta \text{ for every } \eta \in \mathcal C(\Sigma) \} \]
are each open in \( G \). 
Let \( U = U_{\alpha,\beta}  \) and \( V = f \cdot U_\Sigma \). 
Then, as \( G \) has the JEP, there exists \( \sigma \in G \) such that \( U \cap V^\sigma \neq\varnothing \).
Choose a lift \( \tilde \sigma \in \Homeo(M) \) such that each boundary component of  \( \tilde \sigma(\Sigma) \) is in minimal position with \( a \).

Let \( g \in \widetilde G \) be a representative of an element in \( U \cap V^\sigma \) such that \( g|\tilde\sigma(\Sigma) = f^{\tilde\sigma} |\tilde\sigma(\Sigma) \). 
We claim that \( a \) is disjoint from \( \tilde\sigma(\Sigma) \).
To see this, suppose that \( a \cap \tilde\sigma(\Sigma) \neq \varnothing \). 
As \( a \) is in minimal position with the components of \( \partial \tilde\sigma(\Sigma) \) and as \( g  \) maps each component of \( \partial \tilde\sigma(\Sigma) \) to itself, we have that \( g(a) \) and the components of \( \partial \tilde\sigma(\Sigma) \) are in minimal position. 
Therefore, we can project each of \( a \) and \( g(a) \) to \( \cc(\tilde\sigma(\Sigma)) \). 
As \( g(a) \) is homotopic to \( b \) and \( a \cap b = \varnothing \), we have that the distance in \( \cc(\tilde\sigma(\Sigma)) \) between any projection of \( a \) and any projection of \( g(a) \) is bounded by 2. 
This implies that if \( \alpha' \) is a projection of \( a \) to \( \cc(\tilde\sigma(\Sigma)) \), then, as \( g\cdot \alpha' \) is a projection of \( g(a) \), \( d_{\cc(\tilde\sigma(\Sigma))}(\alpha', g\cdot \alpha') \leq 2 \). 
But, as \( g\cdot \alpha' = f^{\tilde\sigma}\cdot \alpha' \), this contradicts the fact that \( d_{\cc(\tilde\sigma(\Sigma))} ( \alpha', g\cdot \alpha') = d_{\cc(\tilde\sigma(\Sigma))}(\alpha', f^{\tilde\sigma} \cdot \alpha') > 2 \); hence, \( a \) is disjoint from \( \tilde\sigma(\Sigma) \). 

As \( g \) fixes \( \tilde\sigma(\Sigma) \) as a set, we can deduce that \( g(a) \) is disjoint from \( \tilde\sigma(\Sigma) \) as well. 
Therefore, after a homotopy, we may assume that both \( a \) and \( b \) are disjoint from \( \tilde\sigma(\Sigma) \), \( g(a) = b \), and \( g \) fixes each component of \( \partial \tilde\sigma(\Sigma) \) pointwise. 
As each component of \( \partial \tilde\sigma(\Sigma) \) is separating and \( \tilde\sigma(\Sigma) \) separates the maximally stable ends of \( M \), there exists a component \( c \) of \( \partial \tilde\sigma(\Sigma) \) separating the maximally stable ends of \( M \).
By construction, \( a \), \( b \), \( c \), and \( g \) satisfy the hypotheses of Lemma~\ref{lem:G-telescoping}, and hence \( a \) and \( b \) are in the same \( \Gamma \)-orbit. 

To finish, note that if \( a \) and \( b \) are not disjoint, then using the fact that \( M \) is weakly self-similar and \( \widetilde G \) is finite index in \( \Homeo(M) \), there exists \( h \in \widetilde G \) such that \( h(a) \) is disjoint from both \( a \) and \( b \).
Then applying the previous case, we find that \( a \) and \( b \) are each in the same \( \Ht \)-orbit as \( h(a) \), and hence in the same \( \Ht \)-orbit as each other. 
\end{proof}


\section{Detecting spaciousness} 
\label{sec:detection}

By the classification of 2-manifolds, a 2-manifold is determined by its genus and the topology of its end space.  
We would therefore like to give a condition for a 2-manifold to be spacious in terms of this data.
This has the added benefit of providing a more concrete process for detecting spacious than is provided by the definition.
We are only able to accomplish this under an additional tameness assumption, and we show that some additional tameness assumption is required for our statement. 
To get started, we first need to recall the preorder on the space of ends introduced by Mann--Rafi \cite{MannRafi}.

Let \( M \) be a 2-manifold.  
Given a subsurface \( \Sigma \) of \( M \) with compact boundary,  let \( \widehat \Sigma \) denote the ends of \( M \) for which \( \Sigma \) is a neighborhood.
Given a homeomorphism \( f \co M \to M \),  let \( \widehat f \co \Ends(M) \to \Ends(M) \) denote the induced homeomorphism. 
Define the binary relation \( \preceq \) on \( \Ends(M) \) as follows: given \( e,e' \in \Ends(M) \),  write \( e \preceq e' \) if for every open neighborhood \( \sU \) of \( e' \) in \( \Ends(M) \) there exists an open neighborhood \( \sV \) of \( e \) in \( \Ends(M) \) and a homeomorphism \( f \in \Homeo(M) \) such that \( \widehat f(\sV) \subset \sU \).  
It readily follows that \( \preceq \) is a preorder on \( \Ends(M) \). 
An end \( \mu \in \Ends(M) \) is \emph{maximal} if \( \mu \preceq e \) implies \( e \preceq \mu \). 
It is easily verified that a maximally stable end of a weakly self-similar 2-manifold is maximal. 
Given two ends  \( e, \mu\in\Ends(M) \), we say that \( e \) is an \emph{immediate predecessor} of \( \mu \) if \( e \prec \mu \) and \( e \prec e' \preceq \mu \) implies \( \mu \preceq e' \). 
Two ends \( e \) and \( e' \)  satisfying \( e \preceq e' \) and \( e' \preceq e \) are said to be of the \emph{same type}\footnote{Though not directly relevant, we note that two ends are of the same type if and only if they are in the same \( \Homeo(M) \)-orbit \cite{MannRafiMonster}. We will only use this fact in the case that \( e \) is stable, which is significantly easier to establish  (see \cite[Lemma~4.2]{VlamisHomeomorphism}).}.

An end \( e \in \Ends(M) \) is \emph{stable} if it admits a neighborhood basis consisting of pairwise-homeomorphic clopen subsets, which we call a \emph{stable neighborhood basis}. 
A clopen subset \( \sU \) in \( \Ends(M) \) is \emph{stable} if it there exists a stable end \( e \) admitting a stable neighborhood basis containing \( \sU \) as an element; we say \( e \) has \( \sU \) as a \emph{stable neighborhood}\footnote{The notion of stable sets and their basic properties were introduced by Mann--Rafi \cite{MannRafi}; however, we will reference \cite{VlamisHomeomorphism} for a concise treatment that is well suited to our purposes.}.

\begin{Def}[Upwardly stable]
\label{def:upward_stable}
The end space \( \Ends \) of a 2-manifold \( M \) is \emph{upwardly stable}  if every maximal end of \( M \) is stable, and if, for any non-maximal end \( e \) of \( M \), there exists a stable clopen subset of \( \Ends \) that contains \( e \)  and is disjoint from the maximal ends of \( M \) (note that \( e \) need not be stable itself). 
A 2-manifold is \emph{upwardly stable} if its end space is. 
\end{Def}

An example of a 2-manifold that fails to be upwardly stable is constructed in Section~\ref{sec:example}. 
Let \( M \) be a weakly self-similar 2-manifold.
If \( M \)  upwardly stable, then \( M \) is tame in the sense of Mann--Rafi \cite{MannRafi}, justifying our claim in the introduction that upward stability is a tameness condition.

We can now give a detection scheme for spaciousness.
Before  restating Theorem~\ref{thm:characterize} from the introduction, we remind the reader that every non-planar weakly self-similar 2-manifold has infinite genus; in particular, it has a non-planar end, and if the manifold is not orientable, it has a non-orientable end (i.e., an end that does not admit an orientable neighborhood in the manifold).

\begin{MainThm2}
Let \( M \) be a upwardly stable weakly self-similar 2-manifold with exactly two maximally stable ends.
Then, \( M \) is spacious if and only if none of the following conditions are met:
\begin{enumerate} [(i)]
\item a maximally stable end of \( M \) has an immediate predecessor with countable \( \Homeo(M) \)-orbit.
\item \( M \) is orientable, infinite genus, and only the maximal ends are non-planar. 
\item \( M \) is not orientable, infinite genus, and only the maximal ends are not orientable. 
\end{enumerate}
\end{MainThm2}

The proof of Theorem~\ref{thm:characterize} will rely heavily on the following properties of stable sets (see \cite[Section~4]{VlamisHomeomorphism}).

\begin{Lem}
\label{lem:structure}
Let \( M \) be a 2-manifold, and let \( \sU \subset \Ends(M) \) be a clopen stable neighborhood of a stable end \( e \).
\begin{enumerate}[(1)]
\item
If \( \sV \) is a clopen neighborhood of \( e \) contained in \( \sU \), then \( \sV \) is homeomorphic to \( \sU \). 
\item
If \( \mathscr W \) is a clopen subset of \( \Ends(M) \) that is homeomorphic to an open subset of \( \sU \ssm \{e\} \), then \( \sU \cup \mathscr W \) is homeomorphic to \( \sU \). \qed
\end{enumerate}
\end{Lem}

\begin{proof}[Proof of Theorem~\ref{thm:characterize}]
Let us consider the backwards implication first, meaning we will assume one of the three conditions above are met. 
Let us consider the first case; the other two are similar.
Let \( e \in \Ends(M) \) be an immediate predecessor of a maximally stable end of \( M \) that has countable \( \Homeo(M) \)-orbit, denoted \( \mathscr O \). 
Fix a tube \( T  \) in \( M \) such that \( \widehat T \cap \mathscr O \neq \varnothing \). 
Let \( T' \) be any tube containing \( T \). 
As \( e \) is an immediate predecessor of a maximally stable end, it follows that \( \widehat T' \cap \mathscr O \)  is finite.
Therefore, there is no element of \( \Homeo(M) \) supported in \( T' \) that displaces \( T \), and as \( T' \) was an arbitrary tube containing \( T \), Lemma~\ref{lem:tube-disjoint} guarantees that \( M \) is not spacious.

Now, we consider the forwards direction: assume that none of the conditions are met. 
We must show that \( M \) is spacious.
If \( M \) has only two ends, then \( M \) is homeomorphic to \( \mathbb S^1 \times \br \) and hence is spacious; we may therefore assume that \( M \) has more than two ends.  
We will in fact establish the slightly stronger fact that \( \Gamma(M) \cdot a = \Homeo(M)\cdot a \) for every simple closed curve \( a \) on \( M \), and hence \( \Gamma(M) \backslash \Homeo(M) \cdot a \) is a singleton for each \( a \). 
If \( a \) does not separate the maximally stable ends, then this follows from Lemma~\ref{lem:G-telescoping}.
So it is left to consider the case where \( a \) separates the maximally stable ends of \( M \).
Recall that \( \Homeo(M) \cdot a \) consists of the curves that separate the maximally stable ends of \( M \).
Let \( c \in \Homeo(M) \cdot a \). 
We will show that \( c \in \Gamma(M) \cdot a \). 

By the hypotheses on \( M \), there exists a curve \( b \) in \( \Homeo(M)\cdot a \) disjoint from both \( a \) and \( c \) such that \( a \) and \( b \), as well as \( c \) and \( b \), co-bound a surface that is infinite genus if \( M \) is infinite genus and contains infinitely many isolated ends if \( M \) does.
We will prove that \( a \) and \( b \) are in the same \( \Gamma(M) \)-orbit.
An identical proof with \( a \) replaced by \( c \) shows that \( c \) and \( b \) are in the same \( \Gamma(M) \)-orbit, and hence \( a \) and \( c \) are in the same \( \Gamma(M) \)-orbit as desired.

Let \( T \) be the tube co-bounded by \( a \) and \( b \), and let \( \sU_b \subset \Ends(M) \) be the clopen subset associated to the component of \( M \ssm b \) containing \( a \).
Let \( e \in \widehat T \). 
As \( e \) is not maximal, by upward stability there exists a non-maximal stable end \( \pi_e \) such that \( e \preceq \pi_e \); moreover, as every end is comparable to the maximally stable ends, \( \pi_e \) can be chosen to lie in \( \sU_b \ssm \widehat T \). 
If \( \pi_e \) is an immediate predecessor of the maximally stable ends, then by assumption, \( \pi_e \) is stable and the \( \Homeo(M) \)-orbit of \( \pi_e \) is uncountable, and hence is perfect (see \cite[Lemma~4.4]{VlamisHomeomorphism}).
If \( \pi_e \) is not an immediate predecessor, then by possibly replacing \( \pi_e \) with an end higher in the preorder, we can assume that \( e \prec \pi_e \). 
In either case, by Lemma~\ref{lem:structure}(2), there exists a clopen neighborhood of \( e \) whose union with any stable neighborhood of 
\( \pi_e \) is again a stable neighborhood of \( \pi_e \). 
Therefore, \( \pi_e \) admits a stable neighborhood \( \sU_e \) containing \( e \). 

Using the compactness of \( \widehat T \), there exists \( e_1, \ldots, e_k \in \widehat T \) such that \( T \subset \sU_1 \cup \cdots \cup \sU_k \), where \( \sU_i = \sU_{e_i} \). 
By Lemma~\ref{lem:structure}(1), \( \sU_i \cap \sU_b \) is homeomorphic to \( \sU_i \) and hence stable with respect to \( e_i \).
In particular, without loss of generality, we may assume that \( \sU_i \subset \sU_b \).
Let \( \pi_i = \pi_{e_i} \).
We will now recursively edit and prune our choices of the \( \pi_i \) so that we may assume that \( \pi_i \) and \( \pi_j \) are incomparable for \( i \neq j \).

Fix \( i \in \{1, \ldots, k\} \), and assume that \( \pi_{j'} \not\preceq \pi_j \ \) for  all \( j < i \) and \( j' \in \{1, \ldots, k\} \). 
Define the set \( J_i \) by \( j \in J_i \) if and only if   \( \pi_j \preceq \pi_i \). 
First suppose that \( \pi_i \) is comparable to an immediate predecessor \( \pi \) of the maximally stable ends (possibly \( \pi = \pi_i \)).
Without loss of generality, we may assume \( \pi \in \sU_b \ssm \widehat T \).
If \( \pi \neq \pi_i \), then replace \( \pi_i \) with \( \pi \), replace \( \sU_i \) with the union of \( \sU_i \) and a stable clopen neighborhood of \( \pi \), and update \( J_i \) if necessary.
Then, using the finite cardinality of \( J_i \), upward stability, and Lemma~\ref{lem:structure}(2), we have that \( \sU_i \) and \( \sU_i \cup \bigcup_{ j \in J_i } \sU_j \) are homeomorphic. 
Hence, by replacing \( \sU_i \) with \( \sU_i \cup \bigcup_{ j \in J_i } \sU_j \) and forgetting each of the \( \pi_j \) for \( j \in J_i \), we have that \( \pi_j \not\preceq \pi_i \) for \( j \neq i \). 
Now, suppose \( \pi_i \) is not comparable to an immediate predecessor of the maximally stable ends.
Using upward stability, there exists a non-maximal stable end \( \pi \in \sU_b \ssm \widehat T \) such that \( \pi_i \prec \pi \), and hence \( \pi_j \prec \pi \) for each \( j \in J_i \).
Replace \( \pi_i \) with \( \pi \), replace \( \sU_i \) with the union of \( \sU_i \) and a stable clopen neighborhood of \( \pi \), and update \( J_i \).
This new \( \pi_i \) is again not comparable to an immediate predecessor of the maximally stable, and hence we can repeat this process with the updated \( \pi_i \).
The choice of \( \pi_i \) stabilizes after finitely many steps. 
We now have that \( \pi_j \prec \pi_i \) for each \( j \in J_i \), and we see that \( \sU_i \cup \bigcup_{ j \in J_i } \sU_j \) is homeomorphic to \( \sU_i \) by Lemma~\ref{lem:structure}(2). 
In particular, as before, we may forget the \( \pi_j \) with \( j \in J_i \) and replace \( \sU_i \) with \( \sU_i \cup \bigcup_{ j \in J_i } \sU_j \).
Therefore, \( \pi_j \not\preceq \pi_i \) for \( j \neq i \). 
After completing this recursion over values of \( i \), no  two of the \( \pi_i \) are comparable.

Observe that if an end \( e' \) is in the stable neighborhood of a stable end \( e \), then \( e' \preceq e \).
Therefore, as the \( \pi_i \) are stable and pairwise incomparable, we must have that \( \pi_i \not\in \sU_j \) whenever \( i \neq j \). 
Therefore, recursively replacing \( \sU_i \) with \( \sU_i \ssm (\sU_{i+1} \cup \cdots \cup \sU_k) \) (which is homeomorphic to \( \sU_i \) by Lemma~\ref{lem:structure}), we may assume that \( \sU_i \) is a stable neighborhood of \( \pi_i \), \( \sU_i \cap \sU_j = \varnothing \) whenever \( i \neq j \), and \( \widehat T \subset \bigcup_{i=1}^k \sU_i \). 
Let \( d \) be a simple closed curve separating the maximally stable ends of \( M \) that is disjoint from \( T \) and co-bounds a tube \( T_b\) with \( b \) such that \( a \subset T_b \) and such that \( \widehat T_b = \sU_1 \cup \cdots \cup \sU_k \). 
Let \( T_a \) be the tube co-bounded by \( a \) and \( d \). 
Note that the hypotheses on the curves \( a \) and \( b \) imply that both \( T_a \) and \( T_b \) have infinite genus if \( M \) is infinite genus and have infinitely many isolated ends if \( M \) does; in particular, the homeomorphism class of \( T_a \) and \( T_b \) are each completely determined by the homeomorphism class of their end spaces.

Letting \( \sV_i = \sU_i \cap \widehat T_a\), we see that \( \sV_i \) is a clopen subset of \( \sU_i \) containing \( \pi_i \) and hence, by stability, homeomorphic to \( \sU_i \). 
Therefore, as \( \widehat T_b \) is the topological disjoint union of the \( \sU_i \) and \( \widehat T_a \) is the topological disjoint union of the \( \sV_i \), we must have that \( \widehat T_a \) and \( \widehat T_b \) are homeomorphic. 
Moreover, the closure of each complement of \( a \) is homeomorphic to the closure of each complement of \( b \).
This allows us to apply the change of coordinates principle to find a homeomorphism \( M \to M \) that fixes \( c \) pointwise and maps \( a \) onto \( b \). 
In particular, \( a \) and \( b \) are in the same \( \Gamma(M) \)-orbit by Lemma~\ref{lem:G-telescoping}; hence, \( M \) is spacious. 
\end{proof}

\subsection{Countable end spaces}
\label{sec:countable}

Let \( M \) be a non-compact 2-manifold with countable end space and that is either (1) planar, (2) orientable with no planar ends, or (3) not orientable with no orientable ends. 
Under these hypotheses, the end space \( \Ends(M) \) of \( M \) is homeomorphic to an ordinal space of the form \( \omega^\alpha \cdot n + 1 \), where \( n \in \bn \) and \( \alpha \) is a countable ordinal.
The ordinal \( \alpha \) is the \emph{Cantor--Bendixson rank} of \( \Ends(M) \) and \( n \) is the \emph{Cantor--Bendixson degree} of \( \Ends(M) \).
We then have that \( M \) is weakly self-similar if and only if the Cantor--Bendixson degree of \( \Ends(M) \) is at most two, with \( M \) being (uniquely) self-similar exactly when the Cantor--Bendixson degree is one.
In particular, \( M \) has exactly two maximally stable ends exactly when the Cantor--Bendixson degree of \( \Ends(M) \) is two. 
Moreover, \( M \) is upwardly stable. 

Given this setup, as a corollary to Lemma~\ref{lem:self-similar_spacious} and Theorem~\ref{thm:characterize}, we give a straightforward characterization of spaciousness in the setting of countable ends. 
Observe that if \( M \) as just described has end space with Cantor--Bendixson degree two, then \( M \) automatically fails conditions (ii) and (iii) on Theorem~\ref{thm:characterize}; moreover, condition (i) is met if and only if the Cantor--Bendixson rank of \( \Ends(M) \) is of the form \( \alpha+1 \) for some successor ordinal \( \alpha \). 
We have established the following:

\begin{Cor}
\label{cor:countable}
Let \( M \) be a non-compact 2-manifold with countable end space and such that \( M \) is either planar, orientable with no planar ends, or not orientable with no orientable ends.
Then, \( M \) is spacious if and only if either (1) the Cantor--Bendixson degree of \( \Ends(M) \) is one or (2) the Cantor--Bendixson degree is two and the Cantor--Bendixson rank of \( \Ends(M) \) is of the form \( \alpha+1 \) for some limit ordinal \( \alpha \).  
\qed
\end{Cor}

\subsection{A tameness condition is necessary}
\label{sec:example}

We now give an example of a 2-manifold \(X\) demonstrating that some kind of tameness assumption on \(M\) is necessary as a hypothesis in Theorem~\ref{thm:characterize}. 
More specifically, the 2-manifold \( X \) constructed below is weakly self-similar with exactly two maximally stable ends,  is not  upwardly stable, is not spacious, and does not meet any of the conditions in Theorem~\ref{thm:characterize}.

We begin by describing a 2-manifold constructed by Mann--Rafi, which we denote \(W\); it is the main building block for our example 2-manifold \(X\). 
The maximal ends of \( W \) are all of the same type, and they form a perfect set;  nevertheless, \( W \) is not self-similar \cite[Section~3]{MannRafiMonster}. 
We describe here the features of \(W\) that are salient to our purposes, and we refer the reader to the paper of Mann--Rafi for full details. Within the end space of \(W\) there exists a countably infinite collection of incomparable types of ends, which we enumerate as \( \{ z_j\}_{j\in\bn} \). The construction of \(W\) begins with the 2-sphere minus a Cantor set, realized as a union pairs of pants indexed by finite binary strings. To each pair of pants we connect sum a certain 2-manifold \(T_{i+1}'\), where \(i\) is the length of the binary string indexing the pair of pants. The important thing about each \(T_{i}'\) is that its end space contains exactly one end of type \(z_j\) for \(j \leq i\) and exactly \(2^{(2^{j-i})}\) ends of type \(z_j\) for \(j > i\).

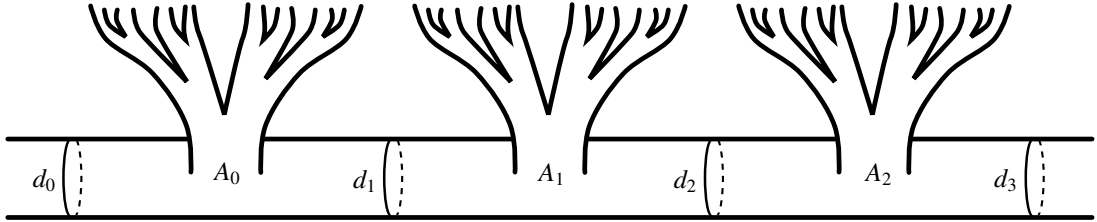
\begin{figure}
 \labellist
 \small\hair 2pt

  \pinlabel {\(d_0\)} [ ] at 30 30
    \pinlabel {\(d_1\)} [ ] at 275 30
      \pinlabel {\(d_2\)} [ ] at 520 30
        \pinlabel {\(d_3\)} [ ] at 765 30

  \pinlabel {\(A_0\)} [ ] at 170 35
    \pinlabel {\(A_1\)} [ ] at 418 35
      \pinlabel {\(A_2\)} [ ] at 668 35

 \endlabellist  
\includegraphics[width=\textwidth]{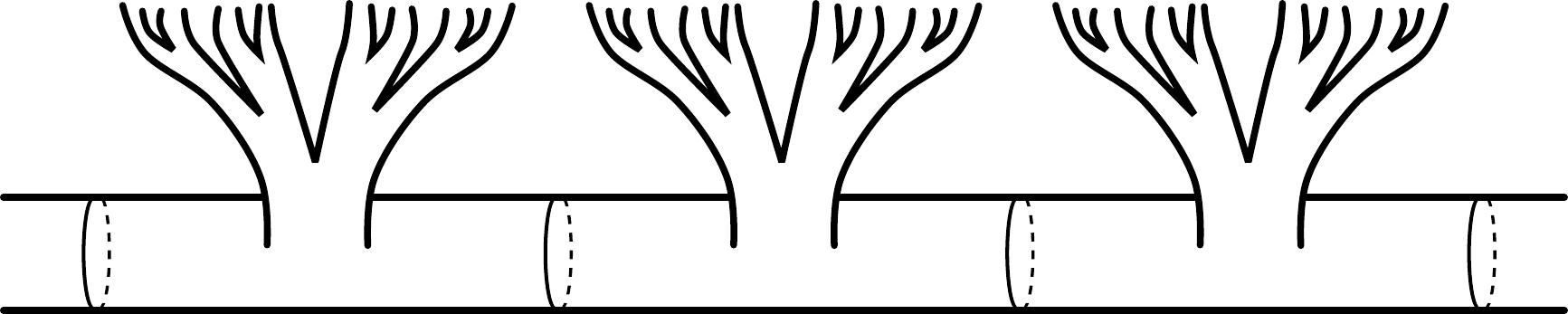}
\caption{The 2-manifold \(X\).}
\label{fig:monsterX}
\end{figure}

Let \(X\) be a 2-manifold built as follows; see Figure~\ref{fig:monsterX}. Take closed annuli \(\{A_i\}_{i \in \Z}\) where each \(A_i\) has boundary components \(\partial_i^+\) and \(\partial_i^-\). Form an open annulus by identifying \(\partial_i^+\) with \(\partial_{i+1}^-\) for all \(i \in \Z\) and call each resulting curve \(d_i\). For each \(i\), connect sum to \(A_i\) a 2-manifold homeomorphic to \(W\); the resulting 2-manifold is \(X\). By abuse of notation, let \(A_i\) now denote the subsurface bounded by \( d_{i-1} \) and \( d_i \) in \( X \). Call \(\Sigma_{i,j}\) the subsurface of \(X\) bounded by \(d_i\) and \(d_j\). Observe that if \(i,j,k,\ell \in \Z\) with \(j-i=\ell-k\), then there exists a homeomorphism of \(X\) carrying \(\Sigma_{i,j}\) to \(\Sigma_{k,\ell}\), and hence \(X\) is weakly self-similar.
Also observe that there is a unique pair of ends \( \mu_- \) and \( \mu_+ \) of \( X \) that are separated by each of the \( d_i \); moreover, we readily see that they are maximally stable ends of \( X \), and hence maximal ends of \(X\). 

We claim these are the only two maximal ends of \( X \). 
Indeed, suppose \( e \) is another such maximal end. 
Then \( e \) is an end of one of the \( A_i \), and hence would be maximal in the end space of \( W \).
However, no maximal end of \( W \) is stable by \cite[Theorem~1.2]{MannRafiMonster}, and hence \( e \) is not a stable end of \( X \). 
We know that \( e \preceq \mu_+ \) and hence, as \( e \) is maximal, \( e \) and \( \mu_+ \) must be of the same type. 
Using the stability of \( \mu_+ \), there exists a homeomorphism \( f \co X \to X \) such that \( f(\mu_+) = e \) (see \cite[Lemma~4.2]{VlamisHomeomorphism}), implying that \( e \) is stable, a contradiction.

By construction we have that \(X\) is not upwardly stable: we have already noted that the maximal ends of the \( A_i \) are not stable, but they are immediate predecessors of the maximally stable ends of \( X \), and hence \( X \) is not upwardly stable.
Our proof that \( X \) is not spacious follows the proof given by Mann--Rafi showing that \( W \) is not self-similar.

\begin{Prop}
\label{prop:monster}
The 2-manifold \(X\) is not spacious.
\end{Prop}

\begin{proof}
We will show that each \(d_i\) is in a different orbit under \(\Ht(X)\), which implies that \( X \) is not spacious. 
Let \( j \) and \( k \) be distinct integers. 
Suppose there exists \(g \in \Ht\) such that \(g(d_j)=d_k\). Let \(T\) be the tubular support of \(g\).
Without loss of generality, we may assume that the boundary curves of \(T\) are \( d_i \) and \( d_\ell \), labelled such that  \(i<j<k<\ell\).
Set \(\ell-k=m\) and set \(\Sigma_{k,\ell}=\Sigma\).
Then, as the forward iterates of \( \Sigma_{j,k} \) under \( g \) are pairwise disjoint and contained in \( \Sigma \), we can choose \( 2m \) disjoint subsurfaces of \( \Sigma \) each with an end space homeomorphic to \( \widehat T_1' \), which we denote by \( E_1 \).
Since \(E_1\) is closed and so is the set of maximal ends of \(\Sigma\), the \(2m\) copies of \(E_1\) are disjoint from some open neighborhood of the set of maximal ends of \(\Sigma\). In particular, there exists \(N\in\bn\) such that the \(2m\) copies of \(E_1\) lie within the union of the end spaces of the first \(N\) levels of the \(m\) copies of \(W\) in \(\Sigma\), where we may assume without loss of generality that \(N\geq 4\). Let this union of end spaces be called \(\widehat{\Sigma}_N\). We will arrive at a contradiction by counting in two different ways the number of ends of type \(z_{N+1}\) in \(\widehat{\Sigma}_N\). First, each copy of \(E_1\) contains exactly   \(2^{(2^N)}\) ends of  type \(z_{N+1}\), so \(\widehat{\Sigma}_N\) contains at least \(2m \cdot 2^{(2^N)}\) ends of type \(z_{N+1}\). On the other hand, counting by levels from 0 to \(N\) in each of the \(m\) copies of \(W\) in \(\Sigma\), we have that the count of ends of type \(z_{N+1}\) in \(\widehat{\Sigma}_N\) is exactly

\begin{align*}
m(2^{(2^N)}+2 \cdot 2^{(2^{N-1})} + \dots + 2^{N} \cdot 2^{(2^{1})})
&=
m\cdot 2^{(2^N)}+m(2^{(2^{N-1}+1)}+\dots+ 2^{(2+N)})\\
&<
m\cdot 2^{(2^{N})}+m\sum_{t=0}^{2^{N-1}+1} 2^t\\
&<
m\cdot 2^{(2^{N})}+m \cdot 2^{(2^{N-1}+2)}\\
&<2m \cdot 2^{(2^N)}
\end{align*}

The first inequality uses the fact that the function \( 2^{2^{N-x}+x} \) is strictly decreasing for increasing positive integral values of \( x < N \); and so these summands are distinct and form a proper subset of the integral powers of \(2\) up to \(2^{(2^{N-1}+1)}\). The second inequality is by the sum formula for a geometric series. The last inequality follows from the fact that \(2^N > 2^{N-1}+2\), since \(N \geq 4\). Since these two different counts for the number of ends of type \(z_{N+1}\) in \(\widehat{\Sigma}_N\) yield a contradiction, we conclude that \(X\) is not spacious.
\end{proof}

So far we have shown that \(X\) is a weakly self-similar with exactly two maximally stable ends, that \( X \) is not upwardly stable, and that \( X \) is not spacious.
Now, it is straightforward to check that \(X\) does not meet any of the three conditions listed in Theorem~\ref{thm:characterize}. The example \(X\) does not meet condition (i), since the only immediate predecessor of the \(\mu_i\) are the maximal ends in the \(A_i\) subsurfaces, and these have uncountable orbit. Next, \(X\) does not meet condition (ii), since by the choice of construction of the \(z_j\), Mann--Rafi arrange either that the whole construction (and therefore \(X\)) is planar, or that the \(z_j\) (which are not maximally stable ends) are non-planar. Finally, \(X\) does not meet condition (iii), since it is orientable.

\bibliographystyle{amsplain}
\bibliography{Bib-jep}

\end{document}